\documentclass[10pt]{amsart}
\usepackage{amsmath}
\usepackage{amsfonts}
\usepackage{amssymb}
\usepackage{amsthm}
\usepackage{url}
\usepackage{tikz-cd}
\usepackage{dsfont}
\usepackage{graphicx}
\usepackage{caption}
\usepackage{subcaption}
\usepackage{comment} 
\usepackage{stmaryrd}
\usepackage{hyperref}
\usepackage{todonotes}
\usepackage{color}
\usepackage{enumerate}
\usepackage{graphicx}

\numberwithin{equation}{section}

\newtheorem{proposition}{Proposition}[section]
\newtheorem{lemma}[proposition]{Lemma}
\newtheorem{theorem}[proposition]{Theorem}
\newtheorem{corollary}[proposition]{Corollary}
\newtheorem{conjecture}[proposition]{Conjecture}

\theoremstyle{definition}
\newtheorem{remark}[proposition]{Remark}
\newtheorem{definition}[proposition]{Definition}
\newtheorem{example}[proposition]{Example}

\DeclareMathOperator{\Bl}{Bl}

\DeclareMathOperator{\Id}{Id}
\DeclareMathOperator{\tr}{tr}

\DeclareMathOperator{\GL}{GL}

\DeclareMathOperator{\Aut}{Aut}

\DeclareMathOperator{\Ch}{Ch}
\DeclareMathOperator{\SL}{SL}

\DeclareMathOperator{\Lie}{Lie}
\DeclareMathOperator{\SU}{SU}

\renewcommand{\L}{\mathcal{L}}
\renewcommand{\phi}{\varphi}
\DeclareMathOperator{\Stab}{Stab}

\DeclareMathOperator{\DF}{DF}
\DeclareMathOperator{\Quot}{Quot}
\DeclareMathOperator{\Hilb}{Hilb}
\DeclareMathOperator{\Hom}{Hom}

\DeclareMathOperator{\W}{W}

\newcommand{\C}{\mathbb{C}}

\newcommand{\Q}{\mathbb{Q}}

\newcommand{\pr}{\mathbb{P}}
\renewcommand{\epsilon}{\varepsilon}
\renewcommand{\H}{\mathcal{H}}

\newcommand{\M}{\mathcal{M}}

\newcommand{\scO}{\mathcal{O}}
\newcommand{\U}{\mathcal{U}}
\newcommand{\X}{\mathcal{X}}

\newcommand{\ddb}{i\partial \bar\partial}
\newcommand{\scL}{\mathcal{L}}

\newcommand{\scV}{\mathcal{V}}
\newcommand{\scZ}{\mathcal{Z}}
\newcommand{\Y}{\mathcal{Y}}

\newcommand{\scU}{\mathcal{U}}
\newcommand{\ddc}{i\partial\bar\partial}

\newcommand{\su}{\mathfrak{su}}
\DeclareMathOperator{\wt}{wt}
\newcommand{\E}{\mathcal{E}}

\title[Stability of fibrations and optimal symplectic connections]{Moduli theory, stability of fibrations and optimal symplectic connections}

\author[Ruadha\'i Dervan and Lars Martin Sektnan]{Ruadha\'i Dervan and Lars Martin Sektnan}
\address{Ruadha\'i Dervan, DPMMS, Centre for Mathematical Sciences, Wilberforce Road, Cambridge CB3 0WB, United Kingdom}
\email{R.Dervan@dpmms.cam.ac.uk}

\address{Lars Martin Sektnan, Institut for Matematik, Aarhus University, 8000, Aarhus C, Denmark}
\email{lms@math.au.dk}

\begin{document}

\begin{abstract}  

K-polystability is, on the one hand, conjecturally equivalent to the existence of certain canonical K\"ahler metrics on polarised varieties, and, on the other hand, conjecturally gives the correct notion to form moduli.  We introduce a notion of stability for families of K-polystable varieties, extending the classical notion of slope stability of a bundle, viewed as a family of K-polystable varieties via the associated projectivisation. We conjecture that this is the correct condition for forming moduli of fibrations. 

Our main result relates this stability condition to K\"ahler geometry: we prove that the existence of an optimal symplectic connection implies semistability of the fibration. An optimal symplectic connection is a choice of fibrewise constant scalar curvature K\"ahler metric, satisfying a certain geometric partial differential equation. We conjecture that the existence of such a connection is equivalent to polystability of the fibration. We prove a finite dimensional analogue of this conjecture, by describing a GIT problem for fibrations embedded in a fixed projective space, and showing that GIT polystability is equivalent to the existence of a zero of a certain moment map.

\end{abstract}

\maketitle

\section{Introduction}\label{sec:intro}

Riemann surfaces are classified by their genus. When the genus is at least two, the automorphism group of each curve is finite and there is a moduli space $\M_g$ parametrising such curves. In genus one, elliptic curves endowed with an ample line bundle have discrete automorphism group and are parametrised by their $j$-invariant, a complex number. There is a unique Riemann surface of genus zero: the Riemann sphere, $\pr^1$. In particular, the associated moduli space is just a point. 

The two basic aims of moduli theory are firstly to construct a space whose points are in bijection with the varieties under consideration, so that one obtains a classification of isomorphism classes of such varieties,  and secondly to understand the behaviour of varieties in families. While the moduli spaces above achieve the first aim, they do not give much understanding of the variation of varieties in families in the case of $\pr^1$. In fact, it is straightforward to give non-trivial examples of families $\pi: X \to B$ such that each fibre $X_b = \pi^{-1}(b)$ is isomorphic to $\pr^1$ for $b\in B$, but such that $X \ncong B \times \pr^1$: one can projectivise any non-trivial vector bundle of rank two over $B$.

Thus to answer the second basic goal of moduli theory in the simplest one-dimensional case, one is led to a different approach. Rather than considering the moduli space of varieties isomorphic to $\pr^1$, which is just a point, one usually studies \emph{vector bundles} over a fixed base $B$. 
Vector bundles, in turn, have a rich theory of their own, and a central object of study in algebraic geometry has been the moduli space of vector bundles over a fixed base.

In higher dimensions, the classical moduli problem is for \emph{polarised} varieties \cite[Section 5.1]{git}, namely varieties endowed with an ample line bundle. In order to obtain a separated (and optimistically, in certain cases, proper) moduli space of polarised varieties, it is well known that one must restrict to varieties satisfying a stability condition. The classical notions arising from geometric invariant theory (GIT) have failed to produce such a moduli space \cite{wang-xu}, and the modern replacement is K-stability \cite{tian-inventiones, donaldson-toric}, which is not a genuine GIT notion \cite{fineross06}. The Yau-Tian-Donaldson conjecture relates K-stability with differential geometry, through the expectation that (at least when $X$ is smooth) K-polystability should be equivalent to the existence of a constant scalar curvature  K\"ahler (cscK) metric in $c_1(L)$ \cite{tian-inventiones, donaldson-toric,yau}. While it is still an open problem to construct a moduli space of K-polystable varieties, it has recently been shown that one can construct a separated moduli space of smooth polarised varieties which admit a cscK metric \cite{moduli, fujiki-schumacher,odaka-compact, odaka-open,lwx1,lwx2}. Just as in the case of Riemann surfaces, however, while this moduli space fully achieves the first goal of moduli theory, it fails to give almost any information about the behaviour of polarised varieties in families, particularly for polarised varieties with large automorphism group.

The goal of the present work is to generalise the perspective of viewing (certain) $\pr^1$-fibrations in terms of vector bundles to much more general fibrations, and to all dimensions. We expect this approach to give a complete solution to the second basic aim of moduli theory. The new idea arises from the bundle theory. When one tries to form moduli spaces of vector bundles, one quickly realises that one must impose a stability condition in order to obtain a separated (and even proper) moduli space; the most appropriate such condition is  slope stability (though Gieseker stability also plays a prominent role). Through the Hitchin-Kobayashi correspondence of Donaldson-Uhlenbeck-Yau \cite{donaldson-surfaces,uhlenbeck-yau}, slope stability of a vector bundle is equivalent to the existence of a canonical metric, in the form of a Hermite-Einstein metric. 

Here we consider the moduli problem for \emph{polarised fibrations} over a fixed base $$\pi: (X,H) \to (B,L),$$ by which we mean that $\pi$ is a flat morphism between projective varieties, $L$ is ample, $H$ is relatively ample, and importantly each fibre $(X_b,H_b)$ is K-polystable. Just as in the bundle theory, in order to obtain a separated and proper moduli space, one expects a stability condition to play a crucial role. The main results are as follows:

\begin{enumerate}[(i)]
\item We introduce a notion of \emph{stability} for polarised fibrations. We conjecture that polystable polarised fibrations over a fixed base form a separated, projective moduli space. This would achieve the second aim of moduli theory.
\item We prove that the existence of a certain canonical metric on the total space of a polarised fibration implies our notion of semistability. The metrics are \emph{optimal symplectic connections}, which we previously introduced as a canonical choice of fibrewise cscK metric on such a fibration \cite[Section 3]{morefibrations}. We conjecture that  existence 
is equivalent to polystability of the fibration.
\item We prove a finite dimensional analogue of the previous conjecture. Namely, we define a GIT problem for fibrations embedded in a fixed projective space, and prove that GIT polystability is equivalent to the existence of a zero of a finite dimensional moment map we introduce on the fibration. 
\end{enumerate}

The above results generalise several of the classical and foundational results in the bundle theory. In particular we show that our notion of semistability extends the notion of slope semistability of a vector bundle $E \to (B,L)$, viewed in terms of the projectivisation $(\pr(E),\scO(1)) \to (B,L)$. The notion of an optimal symplectic connection generalises the notion of a Hermite-Einstein metric, again viewed on the projectivisation \cite[Proposition 3.19]{morefibrations}. Thus $(ii)$ above generalises the classical statement that a vector bundle admitting a Hermite-Einstein metric is slope semistable, and even extends the definition of stability to $\pr^m$-fibre bundles which do not arise as projectivisations of vector bundles. Much of $(iii)$ is motivated by work of Wang \cite{wang-existence-balanced-embeddings,wang-balanced-stable,wang-chow-stability}, and in particular part $(iii)$ above generalises Wang's result that Gieseker stability of a bundle is equivalent to the existence of a certain balanced metric on the bundle.  We now explain our work in more detail, beginning with the stability notion.

\subsection{Stability of fibrations} The definition of K-polystability involves two types of objects, one being a class of degenerations, called test configurations, and the other being a numerical invariant associated to each test configuration \cite{donaldson-toric, tian-inventiones}. A \emph{test configuration} for a polarised variety $(Y,L_Y)$ is a flat family $\C^*\circlearrowright (\Y,\L_{\Y}) \to \C$
 with $\L_Y$ relatively ample, and with fibre $(\Y_t,\L_{\Y_t}) \cong (Y,L_Y)$ for all $t \neq 0$. Each test configuration has an associated \emph{Donaldson-Futaki invariant} $\DF(\Y,\L_{\Y})$, defined using the weight of the $\C^*$-action on the vector spaces $H^0(\Y_0, \L_{\Y_0}^{ r})$ for $r \gg 0$. K-polystability then means that $\DF(\Y,\L_{\Y})\geq 0$ for all test configurations, with equality if and only if the test configuration is a \emph{product}, meaning that $(\Y_0, \L_{\Y_0}) \cong (Y,L_Y)$.

Now suppose that $\pi: (X,H) \to (B,L)$ is a flat family of K-polystable varieties over a projective base. By flatness, the pushforward $\pi_*(H^k)$ is a vector bundle for $k \gg 0$, and $X$ embeds in the projectivisation $X \hookrightarrow \pr(\pi_*(H^k))$. Let $\E \to B \times \C$ be a $\C^*$-degeneration of the bundle $\pi_*(H^k)$, so $\E$ is a flat coherent sheaf over $\C$ and $\E_t \cong \pi_*(H^k)$ for all $t \neq 0$. This induces a degeneration of $X$ inside $(\pr(\E),\scO(1)) \to (B,L) \times \C$, with $\scO(1)$ the relative hyperplane class induced by the projectivisation construction, by setting $\X = \overline{\C^*.X}$ and $\H = \scO(1)|_{\X}$, which is relatively ample over $B \times \C$. $\X$ admits another line bundle, by pulling back $L$ via the natural morphism $$(\X,\H) \to (B,L).$$ We call this setup a \emph{fibration degeneration}.  These arise naturally from a finite-dimensional GIT problem discussed below. 

A fibration degeneration induces a test configuration for $(X,jL+H)$ for $j \gg 0$ by considering $(\X,jL+\H)\to \C$. We show that the Donaldson-Futaki invariant admits an expansion in powers of $j$: $$\DF(\X,jL+\H) = j^n \W_0(\X,\H) + j^{n-1} \W_1(\X,\H) + O(j^{n-2}),$$ with $B$ of dimension $n$. Roughly speaking, the condition $\W_0(\X,\H) = 0$ means that the induced test configuration for a general fibre $(X_b,H_b)$ over $b \in B$ is a product. We say that a fibration is \emph{semistable} if $\W_0(\X,\H) \geq 0$, and when equality holds, we have $\W_1(\X,\H) \geq 0$. In fact the condition that $\W_0(\X,\H) \geq 0$ holds for arbitrary families of K-polystable varieties, and so the crucial part of our definition concerns the term $\W_1(\X,\H).$

The above generalises the notion of slope semistability of a vector bundle, which is typically defined in terms of subsheaves. If $E \to (B,L)$ is a vector bundle over a polarised variety and $F \subset E$ is a coherent subsheaf,  one obtains a $\C^*$-degeneration of $E$ given by $\E \to B \times \C$, flat over $\C$, with $\E_t \cong E$ for all $t \neq 0$ and $$\E_0 \cong F \oplus E/F.$$ In fact all degenerations of $E$ arise this way, provided one considers nested sequences of subsheaves. Projectivising gives a fibration degeneration for $$(\pr(E),\scO(1)) \to (B,L),$$ and work of Ross-Thomas implies that $\W_0(\X,\H)=0$ and  $$\W_1(\X,\H) = c(\mu(E) - \mu(F)),$$ with $\mu$ the slope of the sheaf \cite[Theorem 5.13]{ross-thomas-obstruction} and $c>0$ an unimportant constant. Thus semistability of the fibration implies that $\mu(E)\geq \mu(F)$, implying slope semistability of the bundle. The converse relies on analytic techniques, and is explained in Theorem \ref{HEvsosc}.

\subsection{Optimal symplectic connections} Let $\omega_X\in c_1(H)$ be a closed $(1,1)$-form, with $\omega_X|_{X_b}$ a constant scalar curvature K\"ahler metric for each $b \in B$. We make the smoothness assumption that $\Aut(X_b,H_b)$ has dimension independent of $b \in B$. We briefly recall what it means for $\omega_X$ to be an \emph{optimal symplectic connection} \cite[Definition 3.6]{morefibrations}; full details are given in Section \ref{sec:metrics-stability}. The condition is an equation defined on the total space of $X$, of the form $$ p \left( \Lambda_{\omega_B}\rho_{\H} + \Delta_{\scV}\Lambda_{\omega_B}\mu^*({F_{\H}}) \right) =0. $$ Here $p$ is the natural $L^2$-projection $C^{\infty}(X) \to C^{\infty}_E(X),$ with  $ C^{\infty}_E(X)$ the space of global sections of the vector bundle of holomorphy potentials: functions generating holomorphic vector fields on the fibres and of fibrewise mean-value zero. $\mu^*({F_{\H}})$ is simply the curvature of the Ehresmann connection induced by $\omega_X$, viewed as a two-form with values in Hamiltonian functions via the fibrewise co-moment map $\mu^*$. $\Delta_{\scV}$ is the fibrewise Laplacian, and $\rho_{\H}$ is the horizontal component of the metric on $-K_{X/B}$ induced by $\omega_X$. Finally, $\Lambda_{\omega_B}$ is the horizontal contraction with respect to a base K\"ahler metric $\omega_B \in c_1(L)$. Although not obvious at first glance, the definition generalises the Hermite-Einstein condition on a projective bundle \cite[Proposition 3.19]{morefibrations}. In general, the condition can be viewed as an elliptic partial differential equation on the bundle $E$ \cite[Theorem 4.8]{morefibrations}. The main result of the present work is the following.

\begin{theorem} Suppose a fibration $(X,H) \to (B,L)$ admits an optimal symplectic connection. Then the fibration is semistable.
\end{theorem}

Our proof reduces to a certain analogous finite dimensional problem, mirroring some of the techniques of Donaldson \cite[Theorem 2]{donaldson-semistable}. In fact, we prove a more general result: 

\begin{theorem}\label{intro-lower-bound} For all $j \gg 0$ we have $$\inf_{\omega_X }\|p( \Lambda_{\omega_B}\rho_{\H} + \Delta_{\scV}\Lambda_{\omega_B}\mu^*({F_{\H}}))\|_1\geq - \inf_{(\X,\H)}\frac{j \W_0(\X,\H)+ \W_1(\X,\H)}{\|(\X,\H)\|_{\infty}}.$$ 
\end{theorem}

Here the infimum on the left is taken over all fibrewise cscK metrics $\omega_X \in c_1(H)$, while the infimum on the right is taken over all fibration degenerations. The norm on the left is the $L^1$-norm, while the norm on the right is an $L^{\infty}$-norm which we define. The left hand side is independent of $j$, so the content mostly concerns the situation $\W_0(\X,\H)=0$, in which $\W_1(\X,\H)$ is the term of interest. 

We state here the main conjecture (see Definition \ref{defn:fibstab} for the precise definition of polystability for fibrations):

\begin{conjecture}\label{conj:fibkempfness} A fibration admits an optimal symplectic connection if and only if it is polystable. \end{conjecture}

This generalises the Hitchin-Kobayashi correspondence in the case of projective bundles, proved by Donaldson and Uhlenbeck-Yau \cite{donaldson-surfaces,uhlenbeck-yau}, and is a more precise version of our previous Conjecture \cite[Conjecture 1.1]{morefibrations}.

In some sense the optimal symplectic connection condition appears rather mysterious at first glance. Our notion of stability is natural from the algebro-geometric point of view, however, and we hope that the algebraic geometry gives some motivation for the study of optimal symplectic connections. Moreover, the optimal symplectic connection condition arises naturally in another problem, that of finding constant scalar curvature K\"ahler metrics on the total space of such fibrations $X$. In particular, the existence of an optimal symplectic connection, together with a condition on the base, implies the existence of a constant scalar curvature K\"ahler metric on $(X,jL+H)$ for $j \gg 0$ \cite[Theorem 1.2]{morefibrations}. Our work gives an almost complete converse to that statement: it follows directly from our definitions that an unstable fibration induces a K-unstable total space  $(X,jL+H)$ for $j \gg 0$, which hence cannot admit a constant scalar curvature K\"ahler metric. The analogous problem for the base hypothesis is solved in \cite[Theorem 1.3]{stablemaps}, using the perspective of stable maps.

\subsection{Finite dimensions} The notion of K-stability, and much of the Yau-Tian-Donaldson conjecture, is motivated by an analogous finite dimensional GIT problem \cite{luo,zhang,wang-chow-stability}. Similarly our notion of stability of fibrations is motivated by a GIT problem. For a fibration $(X,H) \to (B,L)$ as before, $X$ embeds into projective space $$X \hookrightarrow \pr(H^0(X,kjL+kH))$$ by choosing a basis, for $k \gg j \gg 0.$ This does \emph{not} embed the base $B$ of the fibration inside projective space, instead one obtains an embedding 
\begin{align*} \phi: B &\to \Hilb \left( \pr(H^0(X,kjL+kH)) \right), \\ b &\to [X_b]\end{align*}
Thus the fibration produces a point $[\phi] \in \Hom(B,\Hilb)$. The space $\Hom(B,\Hilb)$ admits an $\GL$-action arising from the natural action on projective space.

While the scheme $\Hom(B,\Hilb)$ is not proper, we explain how to obtain a projective completion using Grothendieck's Quot scheme. In this manner one-parameter subgroups of $\GL$ (which are the crucial object in GIT via the Hilbert-Mumford criterion) become equivalent to fibration degenerations. 

We show that each orbit $\Phi$ in $\Hom(B,\Hilb)$ admits a natural K\"ahler metric, with respect to which there is a moment map $$\mu: \Phi \to \su^*$$ with respect to its natural $\SU$-action. The following generalises work of Wang in his study of stability of vector bundles \cite[Theorem 1.1]{wang-balanced-stable}; due to differences in the setting we are forced to use a rather different approach to Wang at several points. The following proves the finite dimensional analogue of our main conjecture.

\begin{theorem} $[\phi]$ is polystable if and only if its orbit admits a zero of the moment map $\mu$.
\end{theorem}

\subsection{Outlook} 

Fano varieties play a distinguished role in K-stability. On the algebraic side, roughly speaking, the reason is that K-polystability of a Fano variety is equivalent to K-polystability of Fano varieties with respect to test configurations with \emph{Fano central fibre} \cite[Theorem 1]{li-xu}. 
A subclass of fibration degenerations are those whose central fibre is a Fano fibration, possibly with some singular fibres. 
It seems likely to us that a similar phenomenon occurs:

\begin{conjecture}\label{fano-conjecture} A Fano fibration is polystable if and only if it is polystable with respect to fibration degenerations whose central fibre is a Fano fibration. \end{conjecture}

We next turn to moduli.

\begin{conjecture} There is a separated, quasi-projective moduli space of polystable, possibly singular, fibrations.
\end{conjecture}

This should be the coarse moduli space associated to the moduli functor sending a scheme $S$ to the set of fibrations over $S\times B$, flat over $S$, which for each $s \in S$ are polystable over $B$. We further conjecture that such moduli spaces are projective in many situations of geometric interest; in particular, we conjecture this is so for Fano fibrations. In order to form projective moduli, it is clear one must allow non-smooth and non-flat fibrations, just as in the case of vector bundles. Note that a special case of a Fano fibration is a projective bundle, and in this case the proof of the above conjecture, including projectivity, is a standard (but difficult) part of the bundle theory \cite{simpson}. The reason we focus in this paper on smooth fibrations is that our main result concerns the link with optimal symplectic connections, which for the moment only makes sense for smooth fibrations; the notion of stability makes sense for arbitrary fibrations.

When $B$ is a point, so that the fibration structure is trivial, our theory collapses to the classical case of varieties. In this case, the above conjecture states that one should be able to form quasi-projective moduli of K-polystable varieties. This conjecture, namely the idea that one should be able to use canonical metrics and associated stability conditions to form moduli of compact K\"ahler manifolds, seems to be originally due to Fujiki \cite{fujiki}. In the case of Fano manifolds, this conjecture has been resolved by Odaka \cite{odaka-compact, odaka-open} and Li-Wang-Xu \cite{lwx1,lwx2}. This again leads us to believe the case of Fano fibrations should be most tractable.

One aspect which we do not discuss in the present work is that of examples. Much of our intuition for complex-geometric PDEs comes from algebraic geometry, so we hope that a detailed understanding of stability of fibrations will lead to examples. In any case, our work gives a concrete obstruction to the existence of optimal symplectic connections, which we expect will be checkable in practice. 

\subsubsection*{Outline} In Section \ref{sec:metrics-stability} we recall the definitions of K-stability and optimal symplectic connections. We introduce there the definition of stability of fibrations and explain some of its basic properties. Section \ref{sec:kempf-ness} contains the material on GIT for fibrations, and the related moment map problem. In Section \ref{sec:semistability}, we prove that the existence of an optimal symplectic connection implies semistability of the fibration.

\subsubsection*{Acknowledgements:}The first author thanks Julius Ross and G\'abor Sz\'ekelyhidi for discussions on related topics over the years. We also thank Yuji Odaka and Jacopo Stoppa for helpful comments, and Masafumi Hattori for pointing out an error in our original definition of stability of fibrations in a previous version of the present work. Finally, we would like to thank the anonymous referee for their detailed and very helpful comments. LMS's postdoctoral position is supported by Villum Fonden, grant 0019098. 

\subsection*{Notation} We work throughout over the complex numbers, as our main results concern links with K\"ahler geometry. The definitions of stability of a fibration and other basic results of Section \ref{sec:metrics-stability} go through over an arbitrary algebraically closed field of characteristic zero.

\section{Canonical metrics and stability notions}\label{sec:metrics-stability}
\subsection{Optimal symplectic connections}\label{sec:osc}

Let $\pi: (X,H) \to (B,L)$ be a smooth fibration, so that $\pi$ is a flat morphism between projective varieties with smooth projective fibres. In complex-geometric terms, $\pi$ is thus a holomorphic submersion. Let $\omega_B \in c_1(L)$ be a K\"ahler metric on $B$, and let $\omega_X \in c_1(H)$ be a real closed $(1,1)$-form whose restriction $\omega_b = \left( \omega_X \right)|_{X_b}$ to each fibre $X_b$ is a \emph{constant scalar curvature K\"ahler} (cscK) metric, which means that the scalar curvature $S(\omega_b)$ of $\omega_b$ is constant for each $b \in B$. 

We now assume that the automorphism group $\Aut(X_b,L_b)$ has dimension independent of $b \in B$. This is a genuine condition on the fibration, as flatness only implies that the dimension is an upper semi-continuous function, but holds for example for isotrivial fibrations. The condition can be viewed as a smoothness assumption. When this automorphism group is not finite, a choice of fibrewise cscK metric $\omega_X$ is not unique, and the equation we will be interested in conjecturally determines a canonical choice of such a metric. This condition, which we now describe, is called the \emph{optimal symplectic connection} condition. We refer to \cite[Section 3]{morefibrations} for a more detailed view of the material below.

The form $\omega_X$ defines a splitting $TX = \scV \oplus \H$, where $\H \cong \pi^* TB$ is the $\omega_X$-orthogonal complement to $\scV = \ker d \pi$; in this context $\omega_X$ is usually called a symplectic connection. One similarly obtains a splitting for all tensors, as well as a splitting of functions $$ C^{\infty} \left( X \right) = C^{\infty}_0 \left( X \right) \oplus C^{\infty} \left( B \right),$$ where $C^{\infty}_0 \left( X \right)$ consists of the functions of fibre integral zero. That is, $\phi \in  C^{\infty}_0 (X)$ if the function $\int_{X/B}\phi\omega_X^m$ vanishes, where $m$ is the dimension of the fibres.

The splitting $TX = \scV \oplus \H$ induces an Ehresmann connection for the fibre bundle $X \to B$, which has curvature $F_{\H} \in \Omega^2 \left(B, \textnormal{Ham} \left( \scV \right) \right),$ a $2$-form with values in fibrewise Hamiltonian vector fields on $X$. On each fibre, such a Hamiltonian vector field induces a unique Hamiltonian function which has integral zero over the fibre with respect to the natural volume form. This produces a comoment map $$\mu^* : \textnormal{Ham} \left( \scV \right) \to C^{\infty}_0 (X).$$ Thus, composing with this map, we can view the curvature as $\mu^* (F_{\H})$, a $2$-form on $B$ with values in fibrewise mean zero Hamiltonian functions. 

The form $\omega_X^m$ induces a Hermitian metric on the relative anti-canonical line bundle $\Lambda^m \scV = -K_{X/B}$. This has an associated curvature form, denoted $\rho$, and we let $\rho_{\H}$ denote the horizontal component of $\rho$. 

Via the metrics $\omega_X$ and $\omega_B$, one obtains an $L^2$-inner product on real-valued functions on $X$ by setting $$\langle \phi, \psi\rangle = \int_X \phi\psi \omega_X^m\wedge\omega_B^n.$$ On each fibre $X_b$, there is a vector subspace $E_b\subset C^{\infty}_0(X)$ consisting of functions $\phi \in C^{\infty}_0(X)$ which satisfy $\bar \partial \nabla^{1,0} \phi = 0$, which means that $\phi$ is a \emph{holomorphy potential}. One can form a vector bundle $E \to B$ whose fibre over $b\in B$ is $E_b$ \cite[Section 3.1]{morefibrations}. The space of smooth global sections of $E$ will be denoted $C^{\infty}_E(X)$. We denote by $$p: C^{\infty}(X) \to C^{\infty}_E(X)$$ the natural $L^2$-orthogonal projection.

The final two objects required are the vertical Laplacian $\Delta_{\scV}: C^{\infty}(X) \to C^{\infty}(X) $, which is defined such that $\Delta_{\scV}(\phi)|_{X_b} = \Delta_{\omega_b}(\phi)$, and the horizontal contraction $ \Lambda_{\omega_B}$ defined on purely horizontal forms.

\begin{definition}We say that $\omega_X$ is an \textit{optimal symplectic connection} if 
$$ p \left( \Lambda_{\omega_B}\rho_{\H} + \Delta_{\scV}\Lambda_{\omega_B}\mu^*({F_{\H}}) \right) =0. $$
\end{definition}

\begin{remark} The definition of an optimal symplectic connection may seem mysterious at first glance, and it is worth making the following remarks.
\begin{enumerate}[(i)]
\item Optimal symplectic connections are a generalisation of Hermite-Einstein metrics on holomorphic vector bundles, viewed through the induced fibrewise Fubini-Study metric on the projectivisation \cite[Proposition 3.19]{morefibrations}.
\item We expect that optimal symplectic connections are unique, up to automorphisms, if they exist \cite[Conjecture 1.2]{morefibrations}.
\item Optimal symplectic connections arise naturally from constructions of cscK metrics on $X$ itself \cite[Theorem 1.1]{morefibrations}.
\end{enumerate}
\end{remark}

The main way in which the optimal symplectic connection condition will arise in the present work is as follows. 

\begin{proposition}{\cite[Corollary 4.6]{morefibrations}}\label{prop:scalexpansion} Denote $\omega_j = j\omega_B + \omega_X$. There is a $C^{\infty}$-expansion $$S(\omega_j) = S(\omega_b) + j^{-1}(S(\omega_B) + \Lambda_{\omega_B}\rho_{\H} + \Delta_{\scV}\Lambda_{\omega_B}\mu^*({F_{\H}})) + O(j^{-2}).$$  
\end{proposition} 
Thus the optimal symplectic connection condition asks that the $C^{\infty}_E (X)$ component of the $j^{-1}$ term of $S(\omega_j)$ vanishes. In the present work, we merely need the $L^{\infty}$-statement of the above Proposition.

Rather than working with $\omega_j$ itself, it will be more useful later to work with a related metric. Just as above, the metrics $\omega_X$ and $\omega_B$ induce an $L^2$-orthogonal decomposition $$C^{\infty} (X) = C^{\infty} (B) \oplus C^{\infty}_R (X) \oplus C^{\infty}_E (X).$$ 
\begin{proposition}\cite[Proposition 4.8 and 4.11]{morefibrations}  Suppose $\omega_X$ is a fibrewise cscK metric. Then there is a function $\phi_R \in C^{\infty}_R(X)$ such that, letting $\xi_j = j\omega_B + \omega_X + j^{-1}\ddb \phi_R$, we have $$S(\xi_j) = S(\omega_b) + j^{-1}(\psi_B + p\left( \Lambda_{\omega_B}\rho_{\H} + \Delta_{\scV}\Lambda_{\omega_B}\mu^*({F_{\H}}) \right) + O(j^{-2}),$$ with $\psi_B \in C^{\infty}(B).$
\end{proposition}

\subsection{K\"ahler geometry of the base}

The metric of the base $\omega_B$ is not chosen to have any specific curvature properties. For technical reasons in the arguments later we will, however, need that it can be seen as the solution to an auxiliary PDE on the base.

We begin with a discussion of the base component $\psi_B$ in the expansion of Proposition \ref{prop:scalexpansion}. Since $ \int_{X/B} \Delta_{\scV} ( \phi ) \omega_X^m = 0$ for any $\phi$, the contribution to the $C^{\infty}(B)$-component of the $j^{-1}$ term of $S(\omega_j)$ is the sum of $S(\omega_B)$ and the horizontal component of $\Lambda_{\omega_B}\rho_{\H}$. The latter is related to moduli theory in the following way.

The fibration $\pi : X \to B$ induces a moduli map $q : B \to \mathcal{M}$, with $\M$ the moduli space of cscK manifolds \cite{fujiki-schumacher}\cite[Theorem 1.1]{moduli}. This moduli space carries a K\"ahler metric, the \textit{Weil-Petersson metric} $\omega_{WP}$. The pullback $\alpha = q^*  \omega_{WP} $ of this K\"ahler metric satisfies $$\alpha = \frac{S(\omega_b)}{m+1} \int_{X/B} \omega_X^{m+1} - \int_{X/B} \rho \wedge \omega_X^m,$$ with $\rho$ as above \cite[Theorem 4.4]{moduli}. Then by \cite[Lemma 2.3]{fine2}, \cite[Proposition 4.3]{morefibrations}, the $C^{\infty}(B)$-component of $\Lambda_{\omega_B}\rho_{\H}$ is given by $ - \Lambda_{\omega_B}  \alpha.$ It follows that $\psi_B$ above is given by $$\psi_B =S(\omega_B) - \Lambda_{\omega_B} \alpha.$$ 

The following is motivated by work of Hashimoto \cite{hashimoto19} (see also related work of Zeng \cite{zeng}).

\begin{proposition} For any $r \gg 0$ sufficiently large, there exists a K\"ahler metric $\Omega \in c_1(L^r)$ on $B$ such that $$S(\omega_B) - \Lambda_{\omega_B} \alpha = \Lambda_{\omega_B} \Omega - c_{\Omega},$$ where $c_{\Omega}$ is a topological constant.\end{proposition} 

\begin{proof} The proof is simply an adaptation of \cite[Proposition 1]{hashimoto19} to our setting. Note that if $G$ denotes the Green's operator of the Laplacian of $\omega_B$, then there is a constant $c$ (the average of $S ( \omega_B ) - \Lambda_{\omega_B} \alpha$) such that 
$$ \Delta_{\omega_B} \left( G \left( S ( \omega_B ) - \Lambda_{\omega_B} \alpha \right) \right) = S ( \omega_B ) - \Lambda_{\omega_B} \alpha - c.$$ 
But then for any $r$, we have
\begin{align*} \Lambda_{\omega_B} \left( r \omega_B + \ddb \left( G \left( S ( \omega_B ) - \Lambda_{\omega_B} \alpha \right) \right) \right) &= r n + S ( \omega_B ) - \Lambda_{\omega_B} \alpha + c,
\end{align*}
so that 
$$S ( \omega_B ) - \Lambda_{\omega_B} \alpha = \Lambda_{\omega_B} \Omega - C,$$
where $\Omega = r \omega_B + \ddb G \left( S ( \omega_B ) - \Lambda_{\omega_B} \alpha \right)$ and $C = r n + c.$ So for any $r$, we can solve the required equation. Moreover, $\frac{1}{r} \Omega$ is of the form  $ \omega_B + \frac{1}{r} \ddb \phi$ for some function $\phi$ that does not depend on $r$. Since $\omega_B$ is positive, we can therefore ensure that $\Omega$ is positive when $r$ is sufficiently large.  \end{proof}

The constant $c_{\Omega}$ satisfies $$\int_B   (\Lambda_{\omega_B} \Omega - c_{\Omega} - \hat S_{\alpha})\omega_B^n = 0,$$ where $\hat S_{\alpha}$ satisfies $$\int_B (S(\omega_B) - \Lambda_{\omega_B}\alpha)\omega_B^n = \int_B \hat S_{\alpha} \omega_B^n.$$ We set $$C_{\Omega} = c_{\Omega} + \hat S_{\alpha},$$ which is just the average of $\Lambda_{\omega_B}\Omega$.

In summary, we obtain the following.

\begin{corollary}\label{cor:newscalexp} Let $\eta = \Lambda_{\omega_B}\rho_{\H} + \Delta_{\scV}\Lambda_{\omega_B}\mu^*({F_{\H}})$. Then
\begin{align*}S(\xi_j) &= S(\omega_b) + j^{-1}\left( S(\omega_B) - \Lambda_{\omega_B} \alpha + p \left( \eta \right) \right) + O(j^{-2}), \\
 &= S(\omega_b) + j^{-1} \left( \Lambda_{\omega_B} \Omega - c_{\Omega} +p(\eta) \right) + O(j^{-2}).\end{align*}
\end{corollary}

\subsection{K-stability}

We recall the notion of K-stability for a  normal polarised variety $(Y, L)$ of complex dimension $n$, due to Tian and Donaldson \cite{tian-inventiones,donaldson-toric}. We do not assume that there is any fibration structure.

\begin{definition}\cite[Definition 2.1.1]{donaldson-toric} A \textit{test configuration} $(\Y, \L)$ for $(Y,L)$ is a variety $\Y$ and a line bundle $\L$ on $\Y$, together with 
\begin{enumerate}[(i)] \item a $\C^*$-action on $\Y$ lifting to $\L$,
\item a flat, $\C^*$-equivariant map $p : \Y \to \C$, with $\L$ relatively ample,
\end{enumerate} 
such that each fibre $(\Y_t, \L_t)$ over $t \in \C^*$ is isomorphic to $(Y,L^r)$ for some $r$, called the \emph{exponent} of the test configuration. We say that $(\Y,\L)$ is a \emph{product} test configuration if $(\Y_0,\L_0) \cong (Y,L).$
\end{definition}

Letting $d_k = h^0 (\Y_0, \L_0^k)$, there is an expansion $$ d_k = a_0 k^n + a_1 k^{n-1} + \hdots.$$ Similarly, if $w_k$ denotes the total weight of the $\C^*$-action on $H^0 (\Y_0, \L_0^k)$, then there is an expansion \cite[p. 315]{donaldson-toric} $$ w_k = b_0 k^{n+1} + b_1 k^n + \hdots.$$ 
\begin{definition}\cite[Section 2.1]{donaldson-toric}\label{defn:DF} Let $(\Y, \L)$ be a test configuration for $(Y,L)$. The \textit{Donaldson-Futaki} invariant of $(\Y, \L)$ is defined by $$ \DF (\Y, \L) = \frac{b_0 a_1 - b_1 a_0}{a_0}.$$
\end{definition}

A test configuration $(\Y , \L) \overset{p}{\to} \C$ has a compactification to a degeneration over $\pr^1$, which we will also denote by $(\Y, \L)$, defined by gluing $(\Y, \L)$ to $Y \times \mathbb{P}^1 \setminus \{ 0 \}$ over $p^{-1} (0) \cong Y \times \mathbb{C}^* \subset Y \times \pr^1 \setminus \{ 0 \}.$ When $\Y$ is normal, the Donaldson-Futaki invariant associated to a test configuration for $(Y,L)$ can then be computed in terms of an intersection number on the compactification. By normality, both the canonical divisor $K_{\Y}$ and the relative canonical divisor $K_{\Y / \pr^1} = K_{\Y} - p^*  K_{\pr^1}$ are Weil divisors. It will be useful to define the \emph{slope} of $(Y,L)$ to be $$ \mu(Y,L) =  \frac{- K_Y . L^{n-1}}{L^{n}} .$$ 
\begin{proposition}[{\cite[Theorem 3.2]{odaka2013}, \cite[Proposition 17]{xiaoweiwang12}}]\label{prop:DF} Let $(\Y, \L)$ be a test configuration for $(Y,L)$, of exponent $r$, such that $(\Y, \L)$ is a normal projective variety. Then \begin{align*} \DF (\Y , \L) = \frac{n}{n+1}  \mu (Y, L^r) \L^{n+1} + K_{\Y / \pr^1} . \L^n. 
\end{align*} 
\end{proposition}

\begin{definition}\cite[Definition 2.1.2]{donaldson-toric} Let $(Y,L)$ be a normal polarised variety. Then $(Y,L)$ is 
\begin{enumerate}[(i)] \item \textit{K-semistable} if $\DF (\Y, \L) \geq 0$ for all test configurations $(\Y, \L)$ for $(Y,L)$; 
\item \textit{K-stable} if it is K-semistable and if further $\DF (\Y, \L) = 0$ if and only if $(\Y, \L)$ normalises to the trivial test configuration $(Y\times\C, L)$;
\item \textit{K-polystable} if it is K-semistable and further $\DF (\Y, \L) = 0$ if and only if the normalisation of $(\Y, \L)$ is a product test configuration $(Y \times \C, L^r),$ with a potentially non-trivial $\C^*$-action on $Y$ . 
\end{enumerate}
\end{definition}
Implicit in the above definition is the fact that the normalisation of a test configuration for a normal polarised variety is again a test configuration \cite[Proposition 5.1]{ross-thomas-hm}.

\begin{conjecture}[Yau-Tian-Donaldson]\cite{tian-inventiones, donaldson-toric,yau} A smooth polarised variety $(Y,L)$ admits a constant scalar curvature metric in $c_1 (L)$ if and only if $(Y,L)$ is K-polystable.
\end{conjecture}

\begin{remark}It is expected that the actual stability condition needed for this conjecture to hold in general needs to be stronger than K-polystability. The most likely such candidate is Sz\'ekelyhidi's notion of filtration K-polystability  \cite{szekelyhidi15}, which uses Witt Nystr\"om's intepretation of test configurations in terms of filtrations of the coordinate ring of $(Y,L)$  \cite{wittnystrom} to produce a stronger version of K-polystability.\end{remark}

It will also be important to understand the norm of a test configuration; the most useful norm for us is the minimum norm introduced by the first author and Boucksom-Hisamoto-Jonsson \cite{BHJ,twisted}. For this, passing to a resolution of indeterminacy $X\times\pr^1 \dashrightarrow \X$ if necessary, we may assume that $\X$ admits a morphism to $X \times \pr^1$.

\begin{definition}\label{def:min-norm} We define the \emph{minimum norm} $ \|(\Y,\L) \|_m$ of $(Y,L)$ to be $$\|(\X,\L) \|_m = \frac{\L^{n+1}}{(n+1)} - \L^n.(\L-L).$$ \end{definition}

It follows from  \cite{BHJ,twisted} that a test configuration has norm zero if and only if it normalises to the trivial test configuration. An interpretation of this norm in terms of quantities associated to the $\C^*$-action on $(\Y_0,\L_0)$ is provided in \cite{twisted}.

\subsection{Fibration degenerations}\label{fibrationdegs} Just as K-stability of a polarised variety involves a class of degenerations and a numerical invariant, our notion of stability for fibrations will involve a class of degenerations and a numerical invariant. Here we describe the degenerations, which we call \emph{fibration degenerations}. 

We return to the situation of a flat fibration $\pi: (X,H) \to (B,L)$. By flatness, the pushforward sheaf $ \pi_* \left( H^k \right)$ is the sheaf of sections of a vector bundle $V_k \to B$ for sufficiently large $k$. Indeed, by definition the fibre of $V_k$ over a point $b \in B$ is simply $H^0(X_b,H_b^k)$, whose dimension is independent of $b$ for $k \gg 0$ by our assumptions.  Furthermore, again by relatively ampleness of $H$, $X$ embeds into the projectivisation $\pr(V_k)$, with $\pi$ the restriction to the image of $X$ of the map $\pr (V_k) \to B$. On each fibre, the embedding is the usual Kodaira embedding; this construction can also be performed analytically \cite{lizhangzhao19}.

\begin{definition} A \emph{degeneration} of the vector bundle $V_k$ is a coherent sheaf $\E \to B \times \C$, flat over $\C$, with a $\C^*$-action making $\E \to \C$ a $\C^*$-equivariant map, such that $\E_t \cong V_k$ for all $t \neq 0$. \end{definition} 

These are simply the bundle analogues of test configurations. Given a subsheaf $F$ of $V_k$, one can produce such a bundle degeneration with central fibre $F \oplus V_k/F$ \cite[Remark 5.14]{ross-thomas-obstruction}, explaining the link with the vector bundle theory.

Let $\E$ be a degeneration of $V_k$ for some $k>0$. Then $X$ embeds in $\pr(\E_1) \cong \pr(V_k)$, and one obtains a degeneration of $X$ itself by taking the flat limit of $X$ under the natural $\C^*$-action on $\pr(\E)$. Let $\X = \overline{\C^*.X}$ be the closure. Then $\X$ admits two line bundles: one is the pullback of $L$ via the morphism to $B$, the other is the restriction of the $\scO_{\pr(\E)}(1)$ line bundle arising from the projectivisation construction, which we denote $\H$. By construction, the restriction of $\H$ to a fibre $\X_t$ is isomorphic to $H^k$ for all $t \neq 0$.

\begin{lemma} $(\X,jkL+\H)$ is a test configuration of exponent $k$ for $(X,jL+H)$ for all $k \gg 0$. 
\end{lemma}

\begin{proof} The $\C^*$-action on $\pr(\E)$ induces one on $\X$ itself, which then lifts to $\H$. Thus the only requirement to check is relative ampleness of $jkL+\H$, which is certainly true for $k \gg 0$.\end{proof}

\begin{definition} A \emph{fibration degeneration} for $\pi: (X,H) \to (B,L)$ consists of the data above, namely the family $(\X,\H) \to (B,L)$ arising from a degeneration of the bundle $V_k = \pi_* \left( H^k \right)$ for some $k \gg 0$.\end{definition}

We will explain in Section \ref{sec:kempf-ness} how fibration degenerations arise naturally from considerations in geometric invariant theory.

\begin{example}[Slope stability]\label{ex:stoppa-tenni} One concrete class of fibration degenerations arise through an analogue of Ross-Thomas's notion of slope stability \cite{ross-thomas-hm,ross-thomas-obstruction}; in the easily generalised special case of relative complete intersections in projective bundles, this construction is due to Stoppa-Tenni \cite{stoppa-tenni}. Consider a fibration as above and let $F \subset \pi_*(H^k)$ be a subsheaf. Deformation to the normal cone $$\Bl_{\pr(F) \times 0} \pr( \pi_*(H^k))$$ induces a test configuration for $\pr( \pi_*(H^k))$, which with appropriate polarisations agrees with the projectivisation of the bundle degeneration degenerating $\pi_*(H^k)$ to $(\pi_*(H^k)/ F) \oplus F$. By properties of the blowup, the proper transform of $X \times \C$ agrees with the deformation to the normal cone $\Bl_{\pr(F) \cap X \times 0} X \times \C$ \cite[p. 3]{stoppa-tenni}. Thus the analogue of slope stability for fibration degenerations considers subschemes of $X$ of the form $ \pr(F) \cap X$ with $F \subset \pi_*(H^k)$ a subsheaf for some $k$. 
 \end{example}

\begin{remark}\label{general-fibs} It is not strictly necessary to assume $\pi: (X,H) \to (B,L)$ is flat. Without this assumption, $X$ still embeds in $\pr(\pi_*(H^k))$, which is now the projectivisation of a coherent sheaf. The techniques above still apply to this situation. Stability of more general fibrations should play an important role in the moduli theory of fibrations, however as the focus of the present work is on holomorphic submersions $X \to B$ due to links with K\"ahler geometry, we have chosen to emphasise the technically simpler situation of flat fibrations. \end{remark}

Before turning to the numerical invariant associated to fibration degenerations, which is the remaining part of the definition of stability of fibrations, we need to prove some technical results which will be crucial in later sections. These technical results are not needed for the definition of stability of fibrations; the remaining material needed to define the stability notion is provided in Section \ref{sec:numerical-invariant}.

The following follows from the definitions of the objects involved.

\begin{lemma} For all sufficiently large $k \gg j \gg 0$ we have equalities $$H^0 (X, jkL+kH) = H^0(B, V_k \otimes L^{kj} ) = H^0 \left( \mathbb{P}(V_k), \mathcal{O}(1) \otimes \pi^* \left( L^{jk} \right) \right).$$ 
\end{lemma}

Slightly abusing notation, we will also let $H^0(X, jkL+kH)$ denote  the trivial vector bundle over $B$ with fibre the vector space $H^0(X, jkL+kH)$. There is thus a natural morphism of bundles over $B$ $$H^0(X, jkL+kH) \to V_k  \otimes L^{jk},$$ arising from the identification $H^0 (X, jkL+kH) = H^0(B, V_k \otimes L^{kj})$. In a form that admits generalisation, the map $q: B \to \{pt\}$ induces a map $$q^*H^0(X, jkL+kH) = q^*q_*(V_k  \otimes L^{jk})  \to V_k \otimes L^{kj},$$ where $q^*H^0(X, jkL+kH)$ is the trivial bundle over $B$.

\begin{lemma}\label{globalgen} For sufficiently large $k \gg j \gg 0$, $V_k$ is a quotient $$H^0(X, jkL+kH) \otimes L^{-jk} \to V_k \to 0.$$ 
\end{lemma} 
\begin{proof} We will show this in two steps. First we will show that this is true for every fibre, then we argue that we can choose the parameters uniformly to obtain global generation.

On $B$, we have the exact sequence 
$$ 0 \to \mathfrak{m}_b \cdot V_k \otimes L^{kj} \to V_k \otimes L^{kj}  \to \frac{V_k \otimes L^{kj}}{ \mathfrak{m}_b \cdot V_k \otimes L^{kj}} \to 0,$$ with $ \mathfrak{m}_b$ the ideal sheaf of $b \in B$. Note that $\frac{V_k \otimes L^{kj}}{ \mathfrak{m}_b \cdot V_k \otimes L^{kj}}$ is just the fibre of $V_k \otimes L^{kj}$ at $b$ (namely $H^0(X_b, H^k_b)$), extended by zero. Since $$ H^0 \left( B, V_k  \otimes L^{kj} \right) = H^0 \left( X, jkL+kH \right),$$ we wish to show that $H^0 \left( X, jkL+kH \right)$ surjects onto $H^0 \left( X_b, H^k_b \right)$ when $j,k$ are sufficiently large. From the exact sequence
$$ 0 \to \mathfrak{I}_{X_b} \cdot \left( H \otimes L^{j} \right)^k \to \left( H \otimes L^{j} \right)^k  \to \frac{\left( H \otimes L^{j} \right)^k}{  \mathfrak{I}_{X_b} \cdot \left( H \otimes L^{j} \right)^k } \to 0 $$
on $X$, with $ \mathfrak{I}_{X_b}$ the ideal sheaf of $X_b$, this is guaranteed if 
$$H^1 \left( X,  \mathfrak{I}_{X_b} \cdot \left( H \otimes L^{j} \right)^k \right) =0,$$
which holds by the ampleness of $H \otimes L^{j}$ for large $j$.

The $k,j$ needed above \emph{a priori} depend on $b$. To see that they can be taken independent of $b$, it is enough to show that the function $$b \mapsto \textnormal{dim } H^1 \left( X,  \mathfrak{I}_{X_b} \cdot \left( H \otimes L^{j} \right)^k \right)  $$ is upper semi-continuous in the Zariski topology. This is a reasonably direct application of the semi-continuity theorem \cite[Theorem III.12.8]{hartshorne}, as we now explain.

Consider the graph $\Gamma \subset X \times B$ of $\pi$. The sheaf $H^k \otimes L^{jk}$ is flat over $B$, with respect to the projection $\pi_2$ to the second factor. On $X \times B$ one has an exact sequence of sheaves $$ 0 \to \mathfrak{I}_{\Gamma} \to \mathcal{O}_{X \times B} \to \frac{\mathcal{O}_{X \times B}}{\mathfrak{I}_{\Gamma}} \to 0.$$ Since $\pi$ has connected fibres, $$(\pi_2)_* \left( \frac{\mathcal{O}_{X \times B}}{\mathfrak{I}_{\Gamma}} \right) = \pi_* \left( \mathcal{O}_{X} \right) = \mathcal{O}_B,$$ and hence pushing forward we see that $\mathfrak{I}_{\Gamma}$ is flat over $B$, being one term in an exact sequence with the two others flat. It follows that $\mathfrak{I}_{\Gamma} \cdot ( H^k \otimes L^{jk})$ is also flat over $B$. The restriction of this sheaf to the fibre over $b \in B$ is precisely $\mathfrak{I}_{X_b} \cdot (H^k \otimes L^{jk})$, and thus by the semi-continuity theorem, the dimension map is upper semi-continuous, as required.\end{proof}

\begin{corollary}\label{cor:embeddings1} For all $k \gg j \gg 0$, there are natural embeddings $$X \hookrightarrow\pr(\pi_*(H^k) \otimes L^{jk}) \hookrightarrow \pr(H^0(X,jkL+kH))\times B.$$\end{corollary}

Another consequence is the following.

\begin{lemma}\label{lemma:very-ampleness} There is a $j\gg 0$ such that for all $k \gg j \gg 0$, the line bundle $\scO(1)$ on $\pr(\pi_*(H^k) \otimes L^{jk})$ is very ample. \end{lemma}

\begin{proof} This follows from a standard part of the theory of projective bundles. We may assume $L$ and all of its tensor powers are very ample on $B$, replacing $L$ by a tensor power if not. Thus the line bundle $\scO_{\pr(H^0(X,jkL+kH))}(1) + kL$ is very ample on $\pr(H^0(X,jkL+kH))\times B$, and hence its restriction to $\pr(\pi_*(H^k) \otimes L^{jk})$ is very ample. Since $$\pr(\pi_*(H^k) \otimes L^{jk}) \cong \pr(\pi_*(H^k) \otimes L^{(j+1)k})$$ and $$\scO_{\pr(\pi_*(H^k) \otimes L^{(j+1)k})}(1) = \scO_{\pr(\pi_*(H^k) \otimes L^{(jk})}(1)+kL,$$ the result follows. More precisely, once one takes  $k \gg j \gg 0$ such that Lemma \ref{globalgen} applies, the desired result follows for all $j' > j$. 
\end{proof}

Having realised $V_k$ as a quotient of $H^0 (X, jkL+kH)$ of the trivial bundle over $B$, we perform an analogous procedure for fibration degenerations. Let $(\X, \H) \to B \times \C$ be fibration type degeneration for $(X, H)$. Let $p_2: B \times \C \to \C$ be the projection onto the second factor, and $\pi: \X \to B\times \C$ the natural morphism. Thus by flatness over $\C$, the sheaf $p_{2*}\pi_*(\H^l \otimes L^{jkl})$ is a vector bundle over $\C$ with fibre $H^0(\X_t,  \H_t^l \otimes L^{jl})$ for all $t \in \C$, for $l \gg 0$. Note that on the general fibre $t \neq 0$, this is naturally identified with $H^0(X,  H^{lk} \otimes L^{jkl})$. 

Roughly speaking, the following Lemma states that a fibration degeneration arising from a degeneration of $V_k$ can be viewed as arising from a degeneration of $V_{lk}$ for all $l \gg 0$, and gives a ``family version'' over $\C$ of Lemma \ref{globalgen}.

\begin{lemma}\label{lemma:surjection} For all $l \gg 0$, there is a surjection of sheaves over $B \times \C$ $$(p_2^*p_{2*}\pi_*(\H^l \otimes L^{jkl}))\otimes L^{-jkl} \to  \pi_* \H^{l} \to 0.$$  \end{lemma}

\begin{proof} To obtain the morphism, we again employ the canonical morphism $ g^* g_* \mathcal{F} \to \mathcal{F} $, valid for general morphisms $g$ and sheaves $\mathcal{F}$. From the morphisms $$ \X \overset{\pi}{\to} B \times \C \overset{p_2}{\to} \C,$$  and the bundle $\H^{kl} \otimes L^{jkl} $ over $\X$, the map we are considering is the canonical morphism $$ p_2^* p_{2*} \left( \pi_* \H^{l} \otimes L^{jkl}  \right) \to \pi_* \H^{l} \otimes L^{jkl} .$$
Over each fibre we again obtain a surjection by the same strategy as Lemma \ref{globalgen}, and we wish to apply the semi-continuity theorem to go from a pointwise to global statement. The map to $B$ is still a contraction, so we still have flatness over $B$ of the structure sheaf of the graph. However, this time $\pi_* \H^{kl}$ is not, in general, flat over $B\times \C$, as the fibers may have varying dimension. In order to circumvent this, we first apply flattening stratification to obtain a stratification of $B\times \C$ with $\pi_* \H^{kl}$ flat on each stratum \cite[Theorem 1.6]{kollar-rational}. The fact that the stratification is in the Zariski topology then allows us to argue stratum-wise, and hence the result follows from the same strategy as Lemma \ref{globalgen}. \end{proof}

Note that the statement of this Lemma, restricted to each non-zero fibre, is simply the content of Lemma \ref{globalgen}.

Next, note that the vector bundle $(p_{2} \circ \pi)_*(\H^l \otimes L^{jkl})$ over $\C$ admits a $\C^*$-action, induced from the action on $\H^l$ and hence $\H^l \otimes L^{jkl}$, since $p_2 \circ \pi$ is a $\C^*$-equivariant map. Thus this bundle can be $\C^*$-equivariantly trivialised $$(p_{2} \circ \pi)_*(\H^l \otimes L^{jkl}) \cong H^0(X, klH+jklL) \times \C.$$ Since $\X$ embeds in $(p_{2} \circ \pi)_*(\H^l \otimes L^{jkl})$ by relative ampleness, it follows that $\X$ embeds into $\pr(H^0(X, klH+jklL)) \times \C$ for all $l \gg 0$. From the surjection $$(p_2^*p_{2*}\pi_*(\H^l \otimes L^{jkl}))\otimes L^{-jkl} \to  \pi_* \H^{l} \to 0$$ of sheaves over $B \times \C$, it follows that there is a natural embedding $$\pr(\pi_* \H^{l} \otimes L^{jkl}) \hookrightarrow \pr(p_2^*p_{2*}\pi_*(\H^l \otimes L^{jkl})) \cong \pr(H^0(X, klH+jklL)) \times B \times \C.$$ The form in which we will later make use of this is as follows.

\begin{corollary}\label{fibration-degeneration-embedding} For all $l \gg 0$ there are natural $\C^*$-equivariant embeddings $$\X \hookrightarrow \pr(\pi_* \H^{l} \otimes L^{jkl}) \hookrightarrow \pr(H^0(X, klH+jklL)) \times B \times \C.$$ Moreover, for all $j \gg 0$, there is an $l \gg 0$ such that the $\scO(1)$ bundle on $ \pr(\pi_* \H^{l} \otimes L^{jkl}) $ is relatively very ample over $\C$, giving an embedding $$ \pr(\pi_* \H^{l} \otimes L^{jkl})  \hookrightarrow \pr(H^0(X, klH+jklL)) \times \C.$$
\end{corollary}

The second part of the statement follows from the same ideas as Lemma \ref{lemma:very-ampleness}.

\subsection{Stability of fibrations} \label{sec:numerical-invariant}
We will now associate a numerical invariant to a fibration degeneration, leading to a stability notion for fibrations. Let $(X,H) \to (B,L)$ be a fibration as above, and let $(\X , \H) \to B \times \C$ be a fibration degeneration. 

As with the Donaldson-Futaki invariant of a test configuration, the numerical invariant relevant for fibration degenerations involves the weight function of the $\C^*$-action. For simplicity we assume that $\H$ restricts to $H$ on the general fibre over $\C$; the modifications to the definitions in the general case are straightforward, but complicate notation. The definitions in the general case are the obvious ones; another approach, if the line bundle $\H$ restricts to $H^k$, is to use the $\Q$-line bundle $\frac{1}{k}\H$ in the intersection-theoretic approach to the numerical invariants described below. The scaling property of these invariants implies that approach using $\Q$-line bundles also gives the correct invariants.

The first relevant function is the dimension of the vector space $H^0( X, k(H+jL)),$ given as $$h(j,k) = \dim H^0( X, k(H+jL)) = a_0 (j) k^{n+m} + a_1 (j) k^{n+m-1} + \hdots.$$ The total weight of the $\C^*$-action on $H^0( \X_0, k(\H_0+jL))$ is also a function$$w(j,k) = b_0 (j) k^{n+m+1} + b_1 (j) k^{n+m} + \hdots.$$ 

\begin{lemma}\label{weight-degree} For $k \gg j \gg 0$, the Hilbert function $h(j,k)$ is a polynomial of degree $n+m$ in $k$ and degree $n$ in $j$, and the weight function $w(j,k)$ is a polynomial of degree $n+m+1$ in $k$ and degree $n$ in $j$. \end{lemma}

\begin{proof} Both statements follow from basic intersection theory. For $h(j,k)$, by Riemann-Roch we have $$a_0 (j) = (H+jL)^{n+m} = {n+m \choose n} j^n L^n . H^m + {n+m \choose n-1} j^{n-1} L^{n-1} . H^{m+1} + \hdots,$$ a polynomial of degree $n$ in $j$, and similarly for the lower order terms.

The weight polynomial is given as an Euler characteristic on the total space of the compactification $(\X,\H+jL)$ as $$w(j,k) = \chi(\X,\H+jL) - h^0(X,H+jL),$$ which is again clearly a polynomial in $j$ and $k$ \cite[p. 315]{donaldson-toric} \cite{odaka2013,xiaoweiwang12}. By the same reasoning as above, since $L^{n+1}=0$ as $\dim B =n$, Riemann-Roch for schemes \cite[p. 354]{fulton-intersection} implies that that this is a polynomial of degree $n+m+1$ in $k$ and degree $n$ in $j$.\end{proof}

It follows that the Donaldson-Futaki invariant $$ \DF (\X, \H + j L) = \frac{ b_0 (j) a_1 (j) - b_1 (j) a_0 (j) }{a_0 (j) } $$ is a quotient of a polynomial in $j$ of degree $2n$ by a polynomial of degree $n$, hence admits an expansion $$\DF(\X,jL+\H) = j^n \W_0 ( \X,\H) + j^{n-1} \W_1 (\X,\H) + O(j^{n-2}).$$ 

\begin{definition}\label{defn:fibstab} Let $(X,H) \to (B,L)$ be a fibration such that each fibre $(X_b,H_b)$ is K-polystable. We say that the fibration is \emph{semistable} if $ \W_0 ( \X,\H) \geq 0$, and when equality holds $ \W_1 (\X,\H) \geq 0$.

\end{definition}

To define stability and polystability requires the notion of the norm of a fibration degeneration; the most natural is an analogue of the minimum norm of a test configuration. From its definition, it is clear that the minimum norm admits an expansion $$\|(\X,jL+\H)\|_m = j^n {n+m \choose n}\left(\frac{L^n.\H^m}{m+1} + L^n.\H^m.(\H - H)\right) + O(j^{n-1}).$$ 

\begin{definition} We define the \emph{minimum norm} of the fibration to be $$\|(\X,\H)\|_m = \frac{L^n.\H^m}{m+1} + L^n.\H^m.(\H - H).$$\end{definition}

We can now define stability and polystability of fibrations. In the definitions that follow, we assume the fibration is semistable. 

\begin{definition}\label{defn:fibstab} Let $(X,H) \to (B,L)$ be a semistable fibration. We say that the fibration is \begin{enumerate}[(i)]
\item \emph{stable} if when $ \W_0 ( \X,\H) = 0$ and $ \W_1 (\X,\H)=0$, then the minimum norm $\|(\X,\H)\|_m$ vanishes;
\item \emph{polystable} if when $ \W_0 ( \X,\H) = 0$ and $\W_1 (\X,\H)=0$, then either  the normalisation of $(\X,jL+\H)$ is a product test configuration for $j \gg 0$, or  the minimum norm $\|(\X,\H)\|_m$ vanishes.
\end{enumerate}
\end{definition}

\begin{remark} To make sense of the definitions, it is clear that one need only assume that the \emph{general} fibre of the fibration is K-polystable. Thus, together with Remark \ref{general-fibs}, one can define stability of general fibrations. Similarly, the definition makes sense when the general fibres are simply K-semistable. In fact, we expect that if $ \W_0 ( \X,\H) \geq 0$ for all fibration degenerations, this implies that the general fibre is K-semistable.
\end{remark}

\begin{remark}As explained to us by M. Hattori, it is not true that any fibration degeneration with vanishing minimum norm induces a trivial test configuration; an example and further study of these phenomena will appear in his forthcoming work. An alternative approach to defining stability of fibrations, which also seems reasonable, would be to restrict to fibration degenerations induced by filtrations of $\pi_*(H^k)$ by \emph{torsion-free} subsheaves.
 \end{remark}

The terms $\W_0 ( \X,\H )$ and $\W_1 ( \X,\H )$ can explicitly computed from expansions of the $a_i (j)$ and $b_i (j)$. It will be more enlightening to focus on the case when $\X$ is normal, so that we can use the intersection theoretic formula of Proposition \ref{prop:DF} to calculate the above terms. The Donaldson-Futaki invariant is given as
\begin{align*} \DF(\X,jL+\H) =&  \frac{n+m}{n+m+1}\frac{-K_X.(jL+H)^{n+m-1}}{(jL+H)^{n+m}}(jL+\H)^{n+m+1} \\
& + (jL+\H)^{n+m}.K_{\X/\pr^1}. \end{align*}

Thus expanding in $j$, we see that $$\W_0 (\X,\H) =  {n+m \choose n} \left( \frac{m}{m+1} \left(\frac{-K_{X/B}.L^{n}.H^{m-1}}{L^n.H^m}L^n.\H^{m+1}\right) + L^n.\H^m.K_{\X/B\times \pr^1}\right)$$
and $$\W_1 (\X,\H) = {n+m \choose n-1} \left( C_1(\X,\H) + C_2(\X,\H)+ C_3(\X,\H)+ C_4(\X,\H) \right)$$ where 
\begin{align*}  C_1(\X,\H) &= \frac{m}{m+2}\left( \frac{  -K_{X/B}.L^n.H^{m-1}}{ L^n.H^m}\right) L^{n-1}.\H^{m+2}, \\ 
 C_2(\X,\H) &= -\frac{m}{m+1}\left(\frac{(-K_{X/B}.L^n.H^{m-1})(L^{n-1}.H^{m+1})}{ (L^n.H^m)^2}\right)L^n.\H^{m+1} , \\ 
 C_3(\X,\H) &= \left( \frac{-K_X.L^{n-1}.H^m}{L^n.H^m}\right)L^n.\H^{m+1}, \\  
C_4(\X,\H) &=   L^{n-1}.\H^{m+1}.K_{\X/\pr^1}. 
\end{align*}

\begin{example}[Slope stability] Ross-Thomas's notion of slope stability often leads to explicit constructions of unstable varieties \cite{ross-thomas-obstruction}. As mentioned in Example \ref{ex:stoppa-tenni}, Stoppa-Tenni have previously used slope stability to construct a special case of fibration degenerations \cite{stoppa-tenni}. They calculate the associated Donaldson-Futaki invariant, but in a different limit than the one of interest for stability of fibrations. Their work makes it clear that slope stability could, in principle, be used to calculate the weights $\W_0(\X,\H)$ and $\W_1(\X,\H)$ for certain classes of fibration degenerations; we leave this for future work.
\end{example}

\begin{remark}[Fano fibrations] When $H=-K_{X/B}$ is the relative anti-canonical class, so that $\pi: (X,-K_{X/B}) \to (B,L)$ is a Fano fibration, one calculates that $$\W_0 (\X,\H) =  {n+m \choose n} \left( \frac{m}{m+1} L^n.\H^{m+1} + L^n.\H^m.K_{\X/B\times \pr^1}\right),$$ while $$ {n+m \choose n-1}^{-1}\W_1(\X,\H) = \frac{m}{m+2}L^{n-1}.\H^{m+2} + \frac{1}{m+1}\gamma L^n.\H^{m+1} + L^{n-1}.\H^{m+1}.K_{\X/B\times\pr^1},$$ where $$\gamma = \frac{L^{n-1}.(-K_{X/B})^{m+1}}{L^n.(-K_{X/B})^m}.$$ Note that $-\gamma$ is the degree of the CM line bundle on $B$ induced from the family of K-polystable varieties $\pi: (X,-K_{X/B}) \to (B,L)$, so $\gamma\leq  0$ is non-negative, and vanishes when all fibres are isomorphic (for example, for projective bundles) \cite[Theorem 1.2]{CP}. The formulae simplify further when $\H=-K_{\X/B\times\pr^1}$, which is roughly analogous to special degenerations in the study of Fano varieties. \end{remark} 

The following Lemma implies that it is really condition on $\W_1(\X,\H)$ in Definition \ref{defn:fibstab} that is crucial, as the fibres are K-polystable. 

\begin{lemma}\label{W0term} Let $(\X, \H) \to B \times \C$ be a fibration degeneration for $(X, H) \to (B, L)$. Then $$\W_0 (\X, \H) = {n+m \choose n} L^n  \cdot \DF (\X_b, \H_b) ,$$ a constant multiple of the Donaldson-Futaki invariant of the test configuration $(\X_b, \H_b) \to \C$ for $(X_b, H_b)$, for generic $b$. 
\end{lemma}

\begin{proof} Let $B^0$ be the open dense subset over which the degeneration is flat, produced by flattening stratification \cite[Theorem 1.6]{kollar-rational}. In the Chow ring, $\pi^* [b]$ equals $[\pi^{-1} (b) ]$ for any $b \in B^0,$ by flatness \cite[Theorem 1.25]{3264}. Thus we have $$ \pi^* [b] .\H^{m+1} = \H_b^{m+1}$$ for such a point, and the intersection number $ L^n.\H^{m+1}$ therfore equals the product $ L^n \cdot \H_b^{m+1}$. We emphasise that the notation $ L^n.\H^{m+1}$ is an intersection number on $\X$, while $ L^n \cdot \H_b^{m+1}$ is a product of the intersection numbers $L^n$ and $\H_b^{m+1}$ on $B$ and $\X_b$, respectively. Proceeding similarly with the other terms in the expression for $\W_0 (\X, \H)$, we obtain
\begin{align*} \W_0 (\X, \H) &=  {n+m \choose n} \left( \frac{m}{m+1} \left(\frac{ L^n \cdot \left( -K_{X_b}.H_b^{m-1} \right) }{L^n \cdot H_b^m  } L^n \cdot \H_b^{m+1} \right) + L^n \cdot \H_b^m.K_{\X_b/ \pr^1}\right) \\
&=  {n+m \choose n}  \left(\frac{m}{m+1}  \left(   \mu (X_b, H_b) \cdot L^n \cdot \H_b^{m+1} \right) + L^n \cdot \H_b^m . K_{\X_b/ \pr^1} \right) \\
&= {n+m \choose n}  L^n \cdot \DF (\X_b, \H_b) ,
\end{align*}
as required.
\end{proof}

In particular, since the fibres are K-polystable, we always have $\W_0 (\X, \H) \geq 0$.  Turning to the norm, the same proof furnishes the following.

\begin{lemma} The minimum norm $\|(\X,\H)\|_m$ equals the minimum norm of $(\X_b, \H_b)$ for general $b \in B$.\end{lemma}

By the definition of K-polystability, the condition $\W_0 (\X, \H)=0$ states that the fibration degeneration induces a test configuration for a general fibre which normalises to a product. This observation leads to the following algebraic analogue of the fact that a fibration without fibration automorphisms admits an optimal symplectic connection, namely one from picking $\omega_X$ so that it restricts to the unique cscK metric on each fibre. 

\begin{corollary}\label{cor:noautstable}

Suppose $\Aut (X_b, H_b) =0$ for all $b$, so that each fibre $(X_b,H_b)$ is K-stable. Then $\pi : (X, H) \to (B,L)$ is a stable fibration.

\end{corollary}

\begin{proof} Indeed, since the fibres are K-stable, we have $\DF(\X_b,\H_b) > 0$ for general $b \in B$ provided $\|(\X_b,\H_b)\|_m >0$, again for $b\in B$ general. But by the above, $\W_0(\X,\H)$ is precisely the Donaldson-Futaki invariant of the general fibre, while $\|(\X,\H)\|_m$ agrees with the minimum norm of the general fibre. The proof follows. \end{proof}

%The proof requires some results on the norm of fibration degenerations that will be established as in Section \ref{sec:C0norm}, so we refer the reader to Proposition \ref{prop:proof} in that Section for the proof.  There we will also discuss various norms on the set of fibration degenerations, and in particular in Corollary \ref{cor:c0norm}, we will show that a fibration degeneration has norm zero (with respect to any norm) if and only if its normalisation is the trivial fibration degeneration.

We end by discussing the link between semistability of projective bundles, viewed as fibrations, and slope semistability of the underlying vector bundle. 

\begin{theorem}\label{HEvsosc} Suppose $E \to (B,L)$ is a vector bundle over a smooth polarised variety. Then $\pr(E) \to (B,L)$ is a semistable fibration if and only $E \to (B,L)$ is a slope semistable vector bundle. \end{theorem}

\begin{proof} It is clear from work of Ross-Thomas that semistability of the fibration implies slope semistability of the bundle. Firstly, any subsheaf $F \subset E$ induces a fibration degeneration by projectivising the degeneration of $E$ to  $(E/F)\oplus F$ \cite[Section 5.4]{ross-thomas-obstruction}. Ross-Thomas calculate the Donaldson-Futaki invariant of the associated test configuration $(\pr(\E), jL+\scO(1))$, and it follows from their work that $\W_0(\X,\pr(\E),\scO(1)) = 0$ and the condition $\W_1(\X,\pr(\E),\scO(1)) \geq 0$ is equivalent to the condition $\mu(F) \leq \mu(E)$, with $\mu(F)$ denoting the usual slope of a vector bundle, which is simply the condition required by slope stability.

The reverse direction requires analytic techniques. The slope semistable vector bundle $E$ admits an ``almost Hermite-Einstein metric'' by \cite[Theorem 10.13]{kobayashi}, and in particular admits a sequence of hermitian metrics $h_t$ such that $$\|\Lambda_{\omega_B} F_{h_t} -  \lambda \Id\|_{L^1} \to 0,$$ with $F_{h_t}$ the curvature and $\lambda$ the appropriate topological constant. In the case of projective bundles, we have $$ p \left( \Lambda_{\omega_B}\rho_{\H} + \Delta_{\scV}\Lambda_{\omega_B}\mu^*({F_{\H}}) \right) = \mu^*(F_{h_t})$$ by \cite[Proposition 3.19]{morefibrations}. It then follows from Theorem \ref{intro-lower-bound} that the fibration is semistable, completing the proof of the result.\end{proof}

\begin{remark} We expect the above can be generalised to the analogous statements for stability and polystability, and moreover we expect that the smoothness hypotheses are unnecessary, as would presumably follow from a purely algebro-geometric proof. \end{remark}

\section{A Kempf-Ness Theorem}\label{sec:kempf-ness} The goal of this Section is to understand the finite-dimensional analogue of optimal symplectic connections and stability of fibrations. This will motivate both the definition of stability of a fibration, and the link with differential geometry.

\subsection{Compactifying the space of maps}

We begin with by explaining how fibration degenerations arise in algebraic geometry. Letting $\pi: (X,H) \to (B,L)$ be a flat fibration, the line bundle $kjL+kH$ is very ample for  $k \gg j \gg 0$, producing an embedding $X \hookrightarrow \pr^{N_{j,k}}$. In turn this produces an embedding $$B\hookrightarrow \Hilb(\pr^{N_{j,k}})$$ of $B$ into the Hilbert scheme of subschemes of projective space, via $b \to [X_b] \in \Hilb(\pr^{N_{j,k}})$, where $X_b$ is the fibre of $\pi$ over $b \in B$. Thus our fibration induces a point in the scheme $\Hom(B, \Hilb(\pr^N))$ of morphisms from $B$ to $\Hilb(\pr^{N_{j,k}})$. There is a natural $\GL(N_{j,k}+1)$-action on $\pr^N$, inducing an action on $\Hilb_{h(k)}(\pr^{N_{j,k}})$.  One thus obtains an action on $\Hom(B, \Hilb_{h(k)}(\pr^{N_{j,k}}))$ via right composition $$\phi \to g\circ\phi$$ for $g \in \GL(N_{j,k}+1)$. The space  $\Hom(B, \Hilb_{h(k)}(\pr^{N_{j,k}}))$ is, however, not proper, and the goal of this Section is to produce a natural compactification. Fibration degenerations will then arise from $\C^*$-subgroups of $\GL(N_{j,k}+1)$.

We will use, at several points, the extra structure of the embedding $X \hookrightarrow \pr^{N_{j,k}}$ proved in Section \ref{sec:metrics-stability}. Firstly, for $k \gg j \gg 0$, the pushforward $\pi_*(H^k)$ is a vector bundle by flatness and $X\hookrightarrow \pr(\pi_*(H^k)\otimes L^{jk})$ is an embedding by relative ampleness of $H$. This induces a sequence of embeddings by Corollary \ref{cor:embeddings1} $$X \hookrightarrow \pr(\pi_*(H^k)\otimes L^{jk}) \hookrightarrow \pr(H^0(X,kjL+kH)) \times B,$$ associated to the surjection of bundles over $B$ $$H^0(X,jkL+kH) \to \pi_*(H^k)\otimes L^{jk} \to 0.$$ By Lemma \ref{lemma:very-ampleness}, very ampleness of the $\scO(1)$ line bundle implies that there is also a natural embedding $$\pr(\pi_*(H^k)\otimes L^{jk}) \hookrightarrow \pr(H^0(X,kjL+kH)).$$ Similarly, by Lemma \ref{globalgen} there is a surjection of bundles over $B$. 

We will now consider an abstract subvariety of projective space satisfying similar conditions. Let $X \subset \pr^N$ be a smooth projective linearly normal variety, which admits the structure of a smooth (hence flat) fibration $\pi: X \to B$. Linear normality means that $H^0(X, \scO(1)) \cong H^0(\pr^N, \scO(1)).$ We assume that $\scO_{\pr^N}(1)$ is relatively very ample, so that $\pi_*\scO(1)$ is a vector bundle on $B$ by flatness, and we assume that $\scO_{\pr(\pi_*\scO(1))}(1)\to \pr(\pi_*\scO(1))$ is a very ample line bundle. This induces an embedding $$\phi: B \hookrightarrow \Hilb_{h(k)}(\pr^N),$$ if $h(k)$ is the Hilbert polynomial of $(X_b, \scO(k))$, as well as a surjection $$H^0(\pr^N, \scO(1)) \to \pi_*\scO(1) \to 0$$ of bundles over $B$ just as above. We will also denote by $l(k)$ the Hilbert polynomial of $X \subset (\pr^N,\scO(1))$.

Thus associated to $X$, we have a point $[X] \in \Hilb_{l(k)}(\pr^N)$, and a point $[\phi] \in =\Hom(B, \Hilb_{h(k)}(\pr^N)$. Both $\Hom(B, \Hilb_{h(k)}(\pr^N)$ and $\Hilb_{l(k)}(\pr^N)$ have natural $\GL(N+1)$-actions, induced from the action on $\pr^N$. Let $\Phi \subset \Hom(B, \Hilb_{h(k)}(\pr^N)$ denote the orbit of $[\phi] \in \Hom(B, \Hilb_{h(k)}(\pr^N)$; this is the space we wish to compactify. $\Phi$ has an associated fibration $(\X_{\Phi}, \scO_{\Phi}(1)) \to B \times \Phi$, constructed by pulling back the universal family $(\scZ,\scO_{\scZ}(1)) \to \Hilb_{h(k)}(\pr^N)$ over the Hilbert scheme via the evaluation morphism $$ev: \Hom(B, \Hilb_{h(k)}(\pr^N)) \times B \to \Hilb_{h(k)}(\pr^N).$$ Over each point of $\Phi$, this family is isomorphic to $X \to B$ by definition of $\Phi$.

\begin{lemma} There is a natural embedding $$\Phi \hookrightarrow \Hilb_{l(k)}(\pr^N).$$\end{lemma}

\begin{proof} This follows since $B \to \Hom$ is an embedding by our hypotheses. Over each point $\psi \in \Phi \subset \Hom$, the fibration $\X_{\Phi,\psi} \to B$, defined by taking the fibre over $\psi$, admits a natural embedding in $\pr^N$, since $X \to B$ itself embeds in $\pr^N$ by hypothesis. This defines $\X_{\Phi}$ as a subscheme of $\pr^N \times \Phi$, giving a map $\Phi \to \Hilb_{l(k)}(\pr^N)$ by the universal property of the Hilbert scheme. Since $\phi: B \to \Hom$ is itself an embedding, the same as true for each point of $\Phi$, which implies that the map $\Phi \to \Hilb_{l(k)}(\pr^N)$ we have constructed is also an embedding. \end{proof}

We now produce a natural compactification of $\Phi$. Associated to $\pi: X \to B$ is the surjection $$H^0(\pr^N,\scO(1))\times B \to \pi_*\scO(1) \to 0$$ of vector bundles over $B$. Let $\Quot$ denote the universal quotient of the trivial bundle $H^0(\pr^N,\scO(1))\times B$ over $B$, and let $\scU \to \Quot \times B$ be the associated universal family, so that $[\pi_*\scO(1)] \in \Quot$. The embedding $X \hookrightarrow \pi_*\scO(1)$ then exhibits $[X]$ as a point in $\Hilb(\pr(\scU))$, which is a projective scheme. We perform the same process in families to embed $\Phi \hookrightarrow \Hilb(\pr(\scU)),$ which is then the compactification of interest to us. 

The universal family $q: (\X_{\Phi}, \scO_{\Phi}(1)) \to B \times \Phi$ induces a bundle $q_*(\scO_{\Phi}(1))$ over $B \times \Phi$, which admits a surjection $$H^0(\pr^N,\scO(1)) \times B \times \Phi \to q_*(\scO_{\Phi}(1)) \to 0.$$ By the universal property of the $\Quot$ scheme, this gives a map $\Phi \to \Quot$ such that the pullback of $\U$ to $\Phi$ agrees with $q_*(\scO_{\Phi}(1))$. The embedding $\X_{\Phi} \hookrightarrow \pr(q_*(\scO_{\Phi}(1)))$ induces an embedding  $\X_{\Phi} \hookrightarrow \pr(\scU) \times \Phi$, giving by the universal property of the Hilbert scheme the required embedding \begin{equation} \Phi \hookrightarrow \Hilb(\pr(\scU)).\end{equation} The crucial point is that each point of $\Hilb(\pr(\scU))$ corresponds to a scheme which admits a morphism to $B$, since $\pr(\scU)$ itself admits a morphism to $B$. This is why $\Hilb(\pr(\scU))$ is the important object in the study of fibrations.

We now turn to the natural $\GL(N+1)$-action on $\pr^N$, which induces an action on  $\Phi \subset \Hom(B, \Hilb_{h(k)}(\pr^{N}))$ via right composition. Fibration degenerations, in this way, are automatically equivalent to one-parameter subgroups of $\GL(N+1)$.

\begin{lemma}\label{lemma:fibration-degens-one-ps} A fibration degeneration is equivalent to a one-parameter subgroup of $\GL(N_{j,k}+1)$, acting on $\Hilb(\pr(\U_{j,k}))$ for some $j,k$. \end{lemma}

Here $\U_{j,k}$ is the universal family associated to the Quot scheme constructed above, for arbitrary $j,k$.

\begin{proof} The proof is similar to that of Ross-Thomas for test configurations \cite[Proposition 3.7]{ross-thomas-hm}. Let $[X] \in \Hilb(\pr(\U_{j,k}))$. Given such a $\C^* \hookrightarrow \GL(N_{j,k}+1)$, taking the closure $$\X = \overline{\C^*.X} \subset \Hilb(\pr(\U_{j,k}))$$ inside the proper scheme $\Hilb(\pr(\U_{j,k}))$, it follows that $\X$ admits a a $\C^*$-action and an equivariant morphism to $B \times \C$. The $\scO(1)$-line bundle on $\pr(\U_{j,k})$ restricts to a line bundle on $\X$, hence producing a fibration degeneration.

Conversely, given a fibration degeneration $(\X,\H) \to (B,L)$, Corollary \ref{fibration-degeneration-embedding} produces an embedding $$ \pr(\pi_* \H^{l} \otimes L^{jkl})  \hookrightarrow \pr(H^0(X, klH+jklL)) \times B \times \C,$$ meaning a surjection $H^0(X, klH+jklL)\times B \times \C \to \pi_* \H^{l} \otimes L^{jkl} \to 0$ of bundles over $B \times \C$. This produces a map $\C \to \Quot$, such that $\pi_* \H^{l} \otimes L^{jkl}$ is the pullback of the universal quotient over $\Quot$. The embedding $\X \hookrightarrow \pi_* \H^{l} \otimes L^{jkl}$ exhibits $\X$ as a subscheme of $\pr(\U_{j,k})\times\C$, and equivariance of this construction implies that $\X$ arises through a $\C^*$-subgroup of $\GL(N_{j,k}+1)$.\end{proof}

\subsection{A moment map for fibrations} We now turn to the K\"ahler geometry, and construct a moment map for the $\SU(N+1)$-action on $\Phi$ induced from the $\GL(N+1)$-acion. We first construct a K\"ahler metric on $\Phi$. 

\begin{proposition} $\Phi$ admits a natural $\SU(N+1)$-invariant K\"ahler metric on each $\GL(N+1)$-orbit. \end{proposition}

\begin{proof} It is well known that the Hilbert scheme variety $\Hilb_{h(k)}(\pr^{N})$ admits a natural K\"ahler metric $\Omega^{\Hilb}$, defined first by Zhang \cite[Theorem 1.6]{zhang}. We work in our fixed orbit $\Phi$. Explicitly, letting $Y \subset \pr^{N}$ be a subvariety of dimension $m$ be associated to a point in $\Phi$, one can identify $T_{[Y]}\Hilb_{h(k)}(\pr^{N})$ with $H^0(Y, T\pr^{N}|_Y)$. Then the K\"ahler metric is given as $$\Omega^{\Hilb}_{[Y]}(u,v) = \int_Y \omega_{FS}(u,v) \omega_{FS}^m,$$ where $\omega_{FS}(u,v)$ denotes the pairing between the Fubini-Study metric and the \emph{normal} components of $u,v$ to $Y$  \cite[p. 25]{thomas-notes}. 

Consider the following diagram. 
\[
\begin{tikzcd}
\Phi\times B \arrow[swap]{d}{q} \arrow{r}{ev} & \qquad  \Hilb_{h(k)}(\pr^{N})  \\
\Phi & 
\end{tikzcd}
\]
One then obtains a smooth form on $\Phi$ via $$\Omega = q_*(ev^*\Omega^{\Ch} \wedge \omega_B^n), $$ with $q_*$ denoting the fibre integral. This form is closed by general properties of the fibre integral. So what remains is to show that this is K\"ahler on $\Phi$. 
 
Let $\psi \in \Phi$ induce a fibration $X^{\psi} \subset \pr^N$, via the natural embedding of $\Phi$ into $\Hilb_{l(k)}(\pr^N)$. Via this embedding, a tangent vector $u$ to $\Phi$ at $\psi$ can in particular be understood as an element $H^0(X^{\psi}, T\pr^N|_{X^{\psi}})$ just as above.

We claim that $\omega_B(u,\_)=0$. Since we are working in a fixed $\GL(N+1)$-orbit, a tangent vector at $\psi$ corresponds to an element of the Lie algebra $\mathfrak{gl}(N+1)$. We may then assume $u$ corresponds to a rational element by continuity, and hence integrates to a $\C^*$-action on the orbit $\Phi$. This induces a fibration degeneration by Lemma \ref{lemma:fibration-degens-one-ps}. Then by Proposition \ref{independent-Hamiltonian}, proved in Section \ref{sec:chern-weil} below, the  Hamiltonian associated to this $\C^*$-action vanishes, and hence indeed $\omega_B(u,\_)=0$.

Since $\omega_B(u,\_)=0$, we have \begin{equation}\label{Deligne-formula}\Omega[\psi](u,v) = \int_B \left(\int_{X_{\psi}/B}\Omega^{\Hilb}_{X^{\psi}_b}(u,v)\right)\wedge\omega_B^n,\end{equation} proving positivity since $\Omega_{\Hilb}$ is positive. 

The form $\Omega^{\Hilb}$ is $\SU(N+1)$-invariant, hence since $\SU(N+1)$ acts trivially on $\omega_B$, so is $\Omega$.
\end{proof}

\begin{remark} The above is motivated by a construction due to Wang in the setting of bundles \cite[Lemma 3.3]{wang-balanced-stable}. 
\end{remark}

We next give an algebro-geometric understanding of the K\"ahler metric $\Omega$ just produced, by exhibiting it as the curvature of a metric on an associated line bundle over $\Phi$. This uses the language of Deligne pairings, for which we refer to \cite{moriwaki} for an introduction. Briefly, let $\pi: \scZ \to Y$ be an arbitrary flat projective morphism of integral schemes of relative dimension $m$, and suppose $L_0,\hdots,L_m$ are line bundles on $\scZ$. One can push forward the cycle $$L_0\cdot \hdots \cdot L_m$$ on $\scZ$ to obtain a cycle $$\pi_*(L_0\cdot \hdots \cdot L_m)$$ on $Y$. The \emph{Deligne pairing} upgrades this to a line bundle on $Y$, denoted  $$\langle L_0, \hdots, L_m\rangle,$$ whose first Chern class is the same as of the cycle $L_0\cdot \hdots \cdot L_m$. 

Suppose now that each  $L_i$ is given a hermitian metric $h_i$ with curvature $\omega_i$. One can then induce a \emph{Deligne metric} on the Deligne pairing $\langle L_0, \hdots, L_m\rangle$, which is a hermitian metric denoted by  $$\langle h_0,\hdots, h_m\rangle.$$  When $\pi$ is a smooth morphism, the Deligne metric is a smooth hermitian metric by standard properties of the fibre integral. Its curvature is then given by the fibre integral $$\int_{\scZ/Y} \omega_0 \wedge \hdots\wedge \omega_m,$$ which is a closed $(1,1)$-form on $Y$.

\begin{lemma}Consider the universal fibration $\X_{\Phi} \to \Phi \times B$ over $\Phi$, and let $\scO_(\Phi)(1)$ and $L$ be the line bundles on $\X_{\Phi}$ induced by the universal family and via pullback from $B$. Then $\Omega$ is the curvature of a Deligne metric on the Deligne pairing $$\L_{\Hom} = \langle O(1),\hdots, O(1),L,\hdots, L\rangle,$$ with $m+1$ factors of $O(1)$ and $n$ factors of $L$.
 \end{lemma}

\begin{proof}One sees from Equation \eqref{Deligne-formula} and Zhang's definition of $\Omega^{\Ch}$ through the Deligne pairing that $$\Omega = \int_{\X_{\Phi}/\Phi} \omega_{FS}^{m+1} \wedge \omega_B^n.$$ Thus the result follows from the formula for the curvature of the Deligne pairing.
 \end{proof}

We next turn to the moment map. Recall the moment map $\mu_{FS}: \pr^{N}\to \su_{N+1}$ for the $SU({N}+1)$-action on projective space, defined by $$\mu_{FS}([z_0:\hdots: z_{N}]) = i\frac{z_i\bar z_j}{|z|^2} - \frac{i}{N+1}.$$ This induces a moment map $\mu_{\Hilb}: \Hilb_{h(k)}(\pr^{N}) \to \su_{N+1},$ defined by $$\mu_{\Hilb}(Y) = \int_Y \mu_{FS}\omega_{FS}^m,$$ see for example \cite{wang-balanced-stable}. Strictly speaking as above, the metric is only K\"ahler on any given smooth stratum, but this is irrelevant for our situation. Define a map $$\mu: \Phi \to \su_{N+1}$$ by $$\mu([\psi]) = \int_B \int_{X_{\psi}/B} \mu_{FS} \omega_{FS}^m \wedge \omega_B^n = \int_B \mu_{\Hilb}(X_{\psi}, b)\omega_B^n.$$ Here $\mu_{\Hilb}(X_{\psi}, b) = \mu_{\Hilb}(X_{\psi})(b) \in \su_{N+1}$. Note here that for each fibre $X^{\psi}_b$, the volume $\int_{\X^{\psi}_B}\omega_{FS}^m$ is the same by flatness, meaning that the integral does lie in $ \su_{N+1}$.

\begin{proposition} $\mu$ is a moment map for the $\SU(N+1)$-action on $\Phi$ with respect to the K\"ahler metric $\Omega$. \end{proposition}

\begin{proof} 

Let $g \in \su$ induce $\tilde g \in T\Hilb(\pr^{N})$ via the infinitesimal action. Then $$ \iota_{\tilde g} \Omega  = \int_B  \iota_{\tilde g}\Omega^{\Hilb} \omega_B^n = \int_B \phi^*(d \langle \mu_{\Hilb}, g \rangle )\omega_B^n,$$  where we have used that $ \mu_{\Hilb}$ is a moment map. Then \begin{align*}\int_B \phi^*(d \langle \mu_{\Hilb}, g \rangle )\omega_B^n &= q_*(ev^*(d \langle \mu_{\Hilb}, g \rangle)\wedge \omega_B^n), \\ &= d \langle q_*((ev^*\mu_{\Hilb}) \omega_B^n) , g\rangle.\end{align*} The equivariance of the moment map follows from equivariance of $\mu_{\Hilb}$. 
\end{proof}

\begin{remark} Our proof is almost identical to Wang's analogous proof in the setting of bundles \cite[Lemma 3.4]{wang-balanced-stable}. \end{remark}

The main goal of this Section is to give an algebro-geometric understanding of when $\Phi$ admits a zero of the moment map $\mu$. The key point is that, for each fibration degeneration, one can extend the line bundle produced by the Deligne pairing to all of $\C$, rather than just the $\C^*$ orbit inside $\Phi$. Indeed, as a  fibration degeneration is a morphism of irreducible varieties, the Deligne pairing theory applies to produce a line bundle on $\C$ $$\scL_{\Hom} =\langle \H,\hdots, \H,L,\hdots, L\rangle \to \C.$$ The $\C^*$-action on $\X$ fixes $0 \in \C$ and acts on the one-dimensional vector space $\langle \H,\hdots, \H,L,\hdots, L\rangle_0$ with some weight, which we define to be \begin{equation}\label{weight-def-eqn}\mu(\lambda, [X]) = \wt \langle \H,\hdots, \H,L,\hdots, L\rangle_0.\end{equation} The following weight will be essential in giving an algebro-geometric understanding the moment map $\mu$ algebro-geometrically.

\begin{lemma}\label{GIT-vs-Chern-Weil} The weight $\mu(\lambda, [X])$ is given by $$\mu(\lambda, [X]) = \H^{m+1}.L^n = b_{0,0},$$ where $b_{0,0}$ is the leading order term in $j$ and $k$ of the weight polynomial. \end{lemma}

\begin{proof} The GIT weight is the weight of the action on $\L_{\Ch}|_{[X_0]}$. By a standard compactification argument, one can equivariantly compactify $\L_{\Ch} \to \C$ to $\L_{\Ch} \to \pr^1$, and the GIT weight becomes the degree of this line bundle on $\pr^1$; see for example \cite[Lemma 2.6]{berman}. Since this is a Deligne pairing on $\pr^1$, the definition of its first Chern class implies that this equals $\H^{m+1}.L^n$. In turn, as in Section \ref{sec:metrics-stability}, this equals the term $b_{0,0}$ \cite[Theorem 3.2]{odaka2013}, \cite[Proposition 17]{xiaoweiwang12}. \end{proof}

\subsection{Equivariant Chern-Weil theory}\label{sec:chern-weil} This Section contains a technical result, which will allow us to give a differential-geometric interpretation of the weight $\mu(\lambda, [X])$. This will be crucial in relating stability to the moment map theory, and since this gives a differential-geometric interpretation also for $b_{0,0}$, this will also important in motivating the results of Section \ref{sec:semistability}.

The Hamiltonian condition on a (not necessarily compact) complex manifold $(Y,\omega)$ states for a real holomorphic vector field $v$ on $Y$ that $$dh = \iota_v\omega.$$ The following Lemma states that, on the central fibre of a test configuration, one obtains a Hamiltonian for the natural $\C^*$-action on the smooth locus, which extends as a continuous function on the total space.

\begin{lemma}\label{lemma:Hamiltonians} Let $(\X_0, \H_0,L)$ be the central fibre of a fibration degeneration. Let $\omega_{FS} \in c_1(\H|_0)$ be the restriction of a Fubini-Study metric arising from embedding the fibration degeneration in projective space. Then there is a continous function $$h_j = jh_B + h$$ on $\X_0$ which is a Hamiltonian on the smooth locus of $\X_0$ with respect to $j\omega_B + \omega_{FS}$, with $h$ the restriction of the standard Hamiltonian on projective space. \end{lemma}

\begin{proof} This is a straightforward consequence of the perspective taken in \cite{relative}. Indeed, working on the total space $(\X,jL+\H)$ of the test configuration, it is noted there that any $S^1$-invariant K\"ahler metric $j\omega_B + \omega_{FS} \in c_1(2jL+\H)$ induces a continuous function $h_j$ on $\X$ which is a Hamiltonian on the smooth locus \cite[Lemma 2.17]{relative}; this is proved for arbitrary K\"ahler test configurations. The construction even implies, for example, that if $f: W \to \X$ is any morphism with $W$ smooth, then $f^*h$ is a smooth function on $Z$. We apply this to the K\"ahler metrics $\omega_{FS} + j\omega_B \in c_1(2jL+\H)$, which by restriction gives a sequence of continuous functions $h_j$ on $\X_0$. Since these functions are genuine Hamiltonians in the usual sense on the smooth locus of $\X_0$, by linearity and the definition of a Hamiltonian one sees that they take the form $$h_j = h + jh_B,$$ with $h$ the standard Hamiltonian restricted from projective space. \end{proof}

The proof above does not use the fact that the central fibre $\X_0$ arose from a fibration degeneration. We next show that for fibration degenerations, the Hamiltonians have a special form.

\begin{proposition}\label{independent-Hamiltonian} The Hamiltonians constructed in Lemma \ref{lemma:Hamiltonians} are actually independent of $j$. That is, $h_B=0$. \end{proposition}

\begin{proof} Our strategy is to consider the integral $$f(j) = \int_{\X_0} (jh_B + h_X)^2(j\omega_B + \omega_{FS})^{m+n},$$ which has leading order term in $j$ $$f_0(j) = \int_{\X_0}j^{n+2}h_B^2\omega_B^n\wedge\omega_{FS}^n.$$ If we can show that actually $f_0(j)$ vanishes, then $h_B$ must equal zero on the smooth locus of $\X_0$, and hence by continuity must vanish on $\X_0$.

Donaldson shows that if $Y$ is a subscheme of projective space fixed under a $\C^*$-action on projective space, and if $$\wt H^0(Y,\scO_Y(k)) = b^Y_0 k^{\dim Y} + O(k^{\dim Y-1}),$$ with $\wt$ denoting the total weight of the $\C^*$-action on this vector space, then $$b_0^Y = \int_{Y} h_{FS} \omega_{FS}^{\dim Y},$$ where $h_{FS}$ is the Hamiltonian with respect to the Fubini-Study metric $\omega_{FS}$ on projective space \cite[Section 5.1]{donaldson-semistable}.  Moreover, if $$d_Y(k) =\tr(A_k)^2 = d^Y_0k^{\dim Y + 2} + O(k^{\dim Y+1}),$$ with $A_k$ the trace squared of the weights, then $$d_0^Y = \int_{Y} h_{FS}^2 \omega_{FS}^{\dim Y}.$$ Essentially the same results had earlier been proven by Wang \cite[Theorem 26]{wang-chow-stability}, however Donaldson's strategy adapts more readily to our situation. Clearly our result is a variant of this sort of result, so we briefly recall the strategy of Donaldson's proof.

The main point of Donaldson's argument is that $h$ is a smooth function on projective space, which allows one to integrate over the smooth locus of $\X_0$. In the integral of interest, the scheme $\X_0$ is viewed as a cycle, so the integral is defined as a sum over the integrals over each of the irreducible components, each of which admits a $\C^*$-action. If an irreducible component is not reduced, the integral in turn is defined as a multiple of the integral over the induced reduced scheme. On the smooth locus of (a component of) $\X_0$, the function $h$ is a genuine Hamiltonian, and in order to understand the term $b_0^Y$ Donaldson then constructs an $n+m+1$-dimensional scheme $\Y$ with a line bundle $\L_{\Y}$ which satisfies $$ \L_{\Y}^{m+n+1} = \int_{\Y}c_1(\L_{\Y})^{m+n+1} =-\int_{\X_0} h \omega_{FS}^{m+n}.$$ The scheme $\Y$ is constructed using a fibre bundle construction, so in our situation retains a map $\pi_{\Y}: \Y \to B$. What we wish to emphasise here is that it is not genuinely important that one uses the Fubini-Study metric and its associated Hamiltonian, rather what is important is that the natural Hamiltonian on the smooth locus extends to a continuous function on the entire scheme, which is then automatic for a Hamiltonian restricted from projective space. 

Examining Donaldson's construction in our situation of a fibration degeneration, one sees in this situation that one obtains a line bundle $\H_{\Y}$ such that $$\L_{\Y} = \H_{\Y} + 2j\pi_{\Y}^*L.$$ The construction for line bundles is perhaps more transparent from Ross-Thomas's exposition \cite{ross-thomas-orbifold}, from which it follows that the line bundle associated to $L\to \X_0$ on $\Y$ is of the form $\pi_{\Y}^*L.$

In order to deal with the $d_0^Y$ term, Donaldson similarly constructs an $n+m+2$ dimensional scheme $\pi_{\scZ}: \scZ\to B$ with a line bundle $\H_{\scZ}$ such that \begin{equation}\label{chern-weil-weight-squared}\int_{\X_0} (h+jh_B)^2 (\omega_{FS}+j\omega_B)^{m+n} = \int_{\scZ} (\H_{\scZ}+2j\pi_{\scZ}^*L)^{m+n+2},\end{equation} just as above. Here we are using that in Donaldson's construction, it is not important that one uses the Fubini-Study metric, but rather all one needs is a continuous Hamiltonian extending a genuine one on the smooth locus, whose existence in our situation follows from the previous paragraph. On the right hand side, by the projection formula and the fact that $\dim B = n$ one sees that $$ \int_{\scZ}(2jc_1(\pi_{\scZ}^*L))^{n+2}.c_1(\H_{\scZ})^m = (2j\pi_{\scZ}^*L)^{n+2}.\H_{\scZ}^m = 0,$$ hence on the left hand side of equation \eqref{chern-weil-weight-squared} the $j^{n+2}$-coefficient satisfies $$\int_{\X_0} h_B^2 \omega_{FS}^{m}\wedge \omega_B^n = 0.$$ It follows that $h_B$ is zero on the smooth locus of $\X_0$, hence by continuity vanishes on $\X_0$. \end{proof}

\begin{corollary}\label{cor:chern-weil}  We have $$\int_{\X_0} h \omega_{FS}^{m}\wedge\omega_B^n = b_{0,0}.$$ \end{corollary}

\begin{proof}

As $h_B=0$, an application of Donaldson's result to our situation gives $$\int_{\X_0} h (\omega_{FS}+j\omega_B)^{m+n} = b_0(j).$$ Expanding in $j$ gives $$\int_{\X_0} h \omega_{FS}^m\wedge\omega_B^n = b_{0,0},$$ which is precisely what we wished to prove.\end{proof}

\subsection{Relating stability to moment maps}

We return now to our orbit $\Phi$, which admits a $\GL(N+1)$-action, and let $[\phi] \in \Phi$. By analogy with Geometric Invariant Theory, we make the following definition:

\begin{definition}\label{stability-in-finite-dims} We say that a point $[\phi]$ is 
\begin{enumerate}[(i)] 
\item \emph{semistable} if and only if for all one-parameter subgroups $\lambda$, we have $\mu(\lambda,[\phi]) \geq 0$. 
\item \emph{stable} if and only if for all non-trivial one-parameter subgroups $\lambda$ we have $\mu(\lambda,[\phi]) > 0$. 
\item \emph{polystable} if and only if for all one-parameter subgroups $\lambda \notin \Lie(\Stab([X]))$, we have $\mu(\lambda,[\phi]) > 0$, with equality otherwise. 
\end{enumerate}

\end{definition}

Here, as before, for $\lambda: \C^*\hookrightarrow \SL(N+1)$ a one-parameter subgroup we  denote $$\mu([\phi],\lambda) = \wt \L_{\Hom}([X_0]),$$ following Equation \eqref{weight-def-eqn}.

The criterion in the polystability definition has a concrete interpretation in terms of the fibration $\pi: X \to B$ associated to $[\phi]$. Note that one can modify the below to obtain a statement concerning $\SL(N+1)$ in a natural way. 

\begin{lemma} Let $[\phi] \in \Phi$. Then the stabiliser $\Stab([\phi])$ under the $\GL(N+1)$-action can be identified with $$\Aut(X/B, \scO(1)) = \{g \in \Aut(X,\scO_{\pr^N}(1)): \pi \circ g = \pi\}.$$ \end{lemma}

\begin{proof} We take for granted the standard statement that an element $g \in \SL(N+1)$ fixes $X$ if and only if $g \in \Aut(X,\scO_{\pr^N}(1))$. Thus suppose in addition $g \in \Aut(X/B,\scO(1)) \subset \Aut(X,\scO_{\pr^N}(1))$. Then $\phi \circ g = \phi$, since $g$ fixes each fibre. Hence $\Aut(X/B, \scO(1)) \subset \Stab([\phi])$. 

Conversely, let $g\in \Stab([\phi]) \subset \GL(N+1)$. Note that $\Stab([\phi]) \subset \Stab([X]) = \Aut(X,\scO_{\pr^N}(1))$, since the condition that $g\in \Stab([\phi])$ means that $g \in \Stab([X_b])$ for all $b\in B$. But if $g \in  \Aut(X,\scO_{\pr^N}(1))\setminus \{\Aut(X/B, H)\}$, then there is a $b \in B$ such that $g(y) \notin X_b$ for some $y\in X_b$. But then $g \circ \phi \neq \phi$, a contradiction. \end{proof}

We can now prove the main result of this Section:

\begin{theorem} $[\phi]$  is polystable if and only if the orbit $\Phi$ of $[\phi]$ contains a zero of the moment map $$\mu: \Phi \to \su(N+1).$$ Moreover, in a fixed orbit, zeroes of the moment map are unique up to the action of $\Stab([X])$. 
\end{theorem}

Analogous results hold for stability and semistability. Since we are working on a non-compact space, this does not follow directly from the versions of the Kempf-Ness Theorem in the literature. We will see, however, with an integral class (allowing a rationality statement), that the techniques of the standard proofs apply in our situation. 

\begin{proof} We first show that if the orbit of $[\phi]$ admits a zero of the moment map, then $[\phi]$ is GIT polystable. Since $\mu$ is a moment map, convexity implies that the limit derivative of the associated log norm functional is greater than or equal to zero, with strict positivity for one-parameter subgroups not contained in the stabiliser of $[\phi]$. The derivative of the log norm functional is explicitly $$f(t) = \int_{\lambda(t).[X]}h_{\lambda} \omega_{FS}^m \wedge\omega_B^n,$$ with $h_{\lambda}$ the induced Hamiltonian. Thus the limit as $t$ tends to zero is given by the integral $$\int_{[X_0]}h_{\lambda} \omega_{FS}^m \wedge\omega_B^n.$$ This equals the GIT weight by Corollary \ref{cor:chern-weil} and Lemma \ref{GIT-vs-Chern-Weil}, proving the first part of the statement.

For the reverse direction, we must show that if $[\phi]$ is GIT polystable, then it admits a zero of the moment map.  The cleanest way of showing this is to use the technique exposited by Thomas \cite{thomas-notes}; we briefly recall his argument. One begins by fixing a maximal torus $T \subset \SL(N+1)$ with a basis $\lambda_j$ of one-parameter subgroups. For each such one-parameter subgroup, GIT polystability is shown in \cite[Theorem 4.4]{thomas-notes} to imply the existence of a point on which the Hamiltonian vanishes, using that the action admits a linearisation.  Since the $\lambda_j$ commute, one can find a point at which all Hamiltonians vanish. By equivariance of the moment map, this shows that the moment map itself vanishes at this point, just as in \cite[Theorem 4.4]{thomas-notes}.

Uniqueness is a standard consequence of convexity.\end{proof}

\section{Semistability as a necessary condition}\label{sec:semistability}

In this Section we prove the main result of the paper. Recall that for a fibrewise cscK metric $\omega_X$, we let $\eta = \Lambda_{\omega_B}\rho_{\H} + \Delta_{\scV}\Lambda_{\omega_B}\mu^*({F_{\H}})$, and denote by $p: C^{\infty}(X) \to C_E^{\infty}(X)$ the projection onto fibrewise holomorphy potentials, so that the condition $p(\eta) = 0$ is the optimal symplectic connection condition. 

\begin{theorem}\label{l1-bound}For all $j \gg 0$ we have $$\inf_{\omega_X }\|p(\eta)\|_1\geq - \inf_{(\X,\H)}\frac{j \W_0(\X,\H)+\W_1(\X,\H)}{\|(\X,\H)\|_{\infty}}.$$ 

\end{theorem}

Here the infimum on the left is taken over all fibrewise cscK metrics $\omega_X \in c_1(H)$, while the infimum on the right is taken over all fibration degenerations. The norm used is the $L^{\infty}$-norm of a fibration degeneration, defined as Definition \ref{infty-def} in Section \ref{sec:C0norm}. As before, we use the notation $(\X,\H)$ to signify the data of the fibration degeneration $(\X,\H) \to (B,L)$ for $(X,H) \to (B,L)$.

\begin{corollary} If a fibration admits an optimal symplectic connection, then it is a semistable fibration. \end{corollary}

\begin{proof} Setting $\omega_X$ to be an optimal symplectic connection, the infimum on the left hand side of Theorem \ref{l1-bound} is zero. Thus $\W_0(\X,\H) \geq 0$, and when $\W_0(\X,\H)=0$, it follows that $\W_1(\X,\H) \geq 0$, meaning the fibration is semistable.
\end{proof}

The proof will consist of several steps. Heuristically, we reduce to a finite dimensional problem using Bergman kernels, and then apply convexity of a finite dimensional moment map to obtain the lower bound. The strategy is motivated by work of Donaldson \cite{donaldson-semistable}, however necessarily differs at several key points. In particular, we use the Bergman kernel at a different step, roughly in a more linear way that Donaldson. This  allows us to prove the above, but prevents us from working with more general norms. The $L^{\infty}$-norm has a special subharmonicity property that we expect holds in general, but are unable to establish; see Remark \ref{berndtsson}. When $B$ is a point we recover parts of Donaldson's work, with a slightly different strategy. As we are thus proving a more detailed result than Donaldson, new complications arise from the base $(B,L)$.

\subsection{Bergman kernels}

Any embedding of $X$ in projective space induces a K\"ahler  metric on $X$ by restricting the Fubini-Study metric. The Bergman kernel describes how, using the natural $L^2$-orthogonal embeddings of $X$ inside projective space induced from the hermitian metrics on $kjL+kH$, the Fubini-Study metrics approximate the fixed K\"ahler metrics in $kjc_1(L)+kc_1(H)$. It will be important to use the K\"ahler metrics $$\xi_j = j\omega_B + \omega_X +j^{-1} \ddb \phi_R$$ rather than the K\"ahler metrics $\omega_j = j\omega_B+\omega_X$. Recall from Section \ref{sec:osc} that $\phi_R$ is the fibrewise mean zero function on $X$ chosen such that the expansion in Corollary \ref{cor:newscalexp} holds. We denote by $h_j$ the hermitian metric on $jL+H$ with curvature $\xi_j$.

The $L^2$-inner product on the vector space $H^0(X,k(jL+H))$ is defined for all $k>0$ as $$\langle s, t \rangle = \int_X (s,t)_{h_j^{\otimes k}} \frac{(k\xi_j)^{m+n}}{(m+n)!}.$$ Thus one can take an $L^2$-orthonormal basis $s_i$ of $H^0(X,k(jL+H))$ and define a function $\rho_{k,j} \in C^{\infty}(X)$ by $$\rho_{j,k} = \sum_i |s_i|^2_{h_j},$$ which one checks is independent of chosen orthonormal basis. The following result, due to several authors and beginning with work of Tian \cite{tian-bergman} and Bouche \cite{bouche}, describes the behaviour of the Bergman kernel in the asymptotic limit $k\to \infty$.

\begin{theorem}\cite[Theorem 7.4]{szekelyhidi-book} There is a $C^{\infty}$-expansion $$\rho_{j,k} = 1 + k^{-1}\left(\frac{m+n}{2}\right)S(\xi_j) + O(k^{-2}).$$
\end{theorem}

As explained by Donaldson \cite[Proposition 6]{donaldson-scalar-curvature}, this expansion is uniform in compact families of metrics. Thus one can employ the expansion of the scalar curvature of $\xi_j$ in $j$ from Corollary \ref{cor:newscalexp}, giving $$\rho_{j,k} = 1 + k^{-1}\left(\frac{m+n}{2}\right)(S(\omega_b)+j^{-1}(\Lambda_{\omega_B} \Omega - c_{\Omega} +p(\eta)) + O(j^{-2}) ) + O(k^{-2}),$$ which gives $$\rho_{j,k} = 1 + k^{-1}\left(\frac{m+n}{2}\right)(S(\omega_b)+j^{-1}(\Lambda_{\omega_B} \Omega - c_{\Omega} +p(\eta))) + O(k^{-2}, k^{-1}j^{-2}).$$ The notation $O(\cdot, \cdot)$ means that there is a dependence on both the indicated rates, so that e.g. $f = g + O(k^{-2}, j^{-1})$ means that $|f-g| \leq C k^{-2} + C' j^{-1}$. 

An $L^2$-orthonormal basis $s_0,\hdots, s_{N_{j,k}}$ of $H^0(X,k(jL+H))$ gives an embedding $$\phi_{j,k}: X \to \pr^{N_{j,k}} = \pr(H^0(X,k(jL+H)))$$ by $$x \to [s_0(x):\hdots: s_{N_{j,k}}(x)].$$ The Fubini-Study metric $\omega_{FS}$ on $\pr^{N_{j,k}}$ restricts to a K\"ahler metric $\phi_{j,k}^*(\omega_{FS})$ on $X$. This K\"ahler metric on $X$ is independent of choice of $L^2$-orthonormal basis of $H^0(X,k(jL+H))$, but depends on $j$ and $k$.

The basic important property of the Bergman kernel is the following.

\begin{lemma}\cite[Corollary 7.5]{szekelyhidi-book} We have $$\phi_{j,k}^*(\omega_{FS}) - k\xi_j = \ddc \log \rho_{j,k}.$$ \end{lemma}

At the level of volume forms, by taking a power series expansion of the logarithm function, this implies the following.

\begin{corollary}\label{cor:volforms} We have 
\begin{align*} \phi_{j,k}^*(\omega_{FS})^{m+n} &=  k^{m+n} j^n \binom{m+n}{n} \omega_X^m\wedge \omega_B^n \\
&+ k^{m+n} j^{n-1} \left( \binom{m+n}{n-1}  \omega_X^{m+1} \wedge \omega_B^{n-1} + m \binom{m+n}{n} \omega_X^{m-1} \wedge \ddb \phi_R \wedge \omega_B^n \right)  \\
&+  O(k^{m+n-1} j^{n}, k^{m+n} j^{n-2}). 
\end{align*}
 \end{corollary}
 
 One can use the Bergman kernel again to (asymptotically) replace $k^{m}\omega_X^m\wedge \omega_B^n$ with $\omega_{FS}^m\wedge\omega_B^n.$

\subsection{Reduction to finite dimensions}There are various norms on function spaces on $X$, defined using the available metrics. The most important for us will be the  $L^1$-norm, defined by $$ \|f\|_1 = \int_X |f| \omega_X^m \wedge\omega_B^n.$$ Recall in addition the projection operators $q: C^{\infty}(X) \to C^{\infty}_0(X)$ and $p: C^{\infty}(X) \to C^{\infty}_E(X)$ sending functions to functions with fibre integral zero and to functions which are fibrewise holomorphy potentials respectively. Note that $q$ sends constants to zero. 

Since the K\"ahler metrics $$\xi_j = \omega_j + j^{-1}\ddb \phi_R = j\omega_B + \omega_X + j^{-1} \ddb\phi_R$$ satisfy $$S(\xi_j) = \hat S_b + j^{-1}(S(\omega_B) - \Lambda_{\omega_B} \alpha + p(\eta)) + O(j^{-2}),$$ we have  \begin{align*}q(S(\xi_j)) &= j^{-1}p(\eta) +O(j^{-2}), \\ &=p(S(\omega_j)) + O(j^{-2}) \end{align*} Thus $$ \|p(\eta) \|_1 = j\|q(S(\xi_j)) \|_1 + O(j^{-1}).$$

We now fix a fibration degeneration $(\X,\H) \to (B,L)$. By Corollary \ref{fibration-degeneration-embedding}, for all $k \gg j \gg 0$ this degeneration can be realised through a $1$-parameter subgroup of $\GL(N_{j,k}+1)$, hence embedding $\X$ inside $\pr(H^0(X,k(jL+H))) \times \C$. Thus in what follows we will be most concerned in the asymptotic behaviour. Perhaps applying a unitary change of basis, we assume the action is diagonal, hence admits a Hamiltonian function $h$ with respect to the Fubini-Study metric of the form $$h = \sum_i \frac{h_i |z_i|^2}{\sum_i |z_i|^2}.$$ The fibration degeneration has associated numerical invariants $h(j,k)$ and $w(j,k)$ given by the Hilbert polynomial and weight polynomial respectively, with $h(j,k)$ is of degree $n+m$ in $k$ and degree $n$ in $j$, and with $w(j,k)$ of degree $n+m+1$ in $k$ and $n$ in $j$ by Lemma \ref{weight-degree}. The leading order terms of $h(j,k)$ and $w(j,k)$ in $k$ are denoted $a_0(j)$ and $b_0(j)$ respectively. In what follows, $\|f\|_{\infty}$ denotes the $L^{\infty}$-norm of a function, namely the maximum of its absolute value.

\begin{proposition}\label{prop:reduction} The function $ \|p(\eta)\|_1 \left\|\sum_i h_i |s_i|^2_{h^k_X \otimes h^{kj}_{B}} - \frac{w(j,k)}{k^{m+n} a_0(j)}\right\|_{\infty}$ is bounded below by \begin{align*} &-\binom{m+n}{n}^{-1} k^{1-m-n} j^{1-n}  \int_X \left(\sum_ i \frac{h_i|z_i|^2}{\sum_i |z_i|^2} - \frac{w(j,k)}{h(j,k)}\right) \omega_{FS}^{m+n}  \\ &- \binom{m+n}{n}^{-1}  k^{-m-n} j^{-n} \int_X \left(\sum_i h_i \frac{|z_i|^2}{\sum_i |z_i|^2}- \frac{w(j,k)}{h(j,k)}\right)(\Lambda_{\omega_B}\Omega -C_{\Omega})  \omega_{FS}^{m+n}. \end{align*}\end{proposition}

\begin{proof} The proof begins with an examination of the classical balancing functional, which takes the form $$\mathcal{B}(\phi_{j,k}(X)) =  \int_X \left(\sum_ i \frac{h_i|z_i|^2}{\sum_i |z_i|^2} - \frac{w(j,k)}{h(j,k)}\right) \frac{\omega_{FS}^{m+n}}{(m+n)!},$$ with $\phi_{j,k}: X \to \pr(H^0(X,kjL+kH))$ the Kodaira embedding. The function $h$ is homogeneous in the $|z_i|^2$, so we can write $$h = \sum_i \frac{h_i |z_i|^2}{\sum_i |z_i|^2} =  \sum_i \frac{h_i |s_i|_{h_B^j \otimes h_X}^2}{\sum_i |s_i|_{h_B^j \otimes h_X}^2} =  \sum_i \frac{h_i |s_i|_{h_B^j \otimes h_X}^2}{\rho_{j,k}}.$$ When clear from context, we will from now on omit the the norm subscript in $|s_i|_{h_B^j \otimes h_X}$. 

Let $\hat S_j$ be the constant satisfying $\int_X \hat S_j \omega_j^{m+n} = \int_X S(\omega_j) \omega_j^{m+n}.$ Then, using $$\left(\frac{m+n}{2}\right)\hat S_j = \frac{a_1(j)}{a_0(j)}$$ one computes $$  \frac{1}{h(j,k)} = \frac{1}{a_0(j)k^{m+n}}\left(1 - k^{-1}\left(\frac{m+n}{2}\right)\hat S_j\right) + O(k^{-m-n-2}).$$ Using this together with the Bergman kernel expansion for $\rho_{j,k}$, we see \begin{align*} &\mathcal{B}(\phi_{j,k}(X)) \\ =& \int_X \left(\sum_i h_i |s_i|^2(1 - k^{-1}\left(\frac{m+n}{2}\right)S(\xi_j)) - \frac{w(j,k)}{k^{m+n} a_0(j)}(1 - k^{-1} \left(\frac{m+n}{2}\right)\hat S_j)\right) \frac{(k\xi_j)^{m+n}}{(m+n)!} \\&+ O(k^{m+n-1} j^n).\end{align*} Since $w(j,k) = \sum_i h_i$, the $L^2$-orthonormality of the $s_i$ implies $$ \int_X \sum_i h_i |s_i|^2 \frac{(k\xi_j)^{m+n}}{(m+n)!}= \int_X \frac{w(j,k)}{k^{m+n} a_0(j)} \frac{(k\xi_j)^{m+n}}{(m+n)!} + O(k^{m+n} j^n, k^{m+n+1}j^{n-1}).$$ Thus since the function $\frac{w(j,k)}{k^{m+n} a_0(j)}$ is constant, this gives \begin{align*}\mathcal{B}(\phi_{j,k}(X)) &= k^{-1}\int_X \left(-\sum_i h_i |s_i|^2 + \frac{w(j,k)}{k^{m+n} a_0(j)}\right) S(\xi_j) \frac{(k\xi_j)^{m+n}}{(m+n)!} +O(k^{m+n-1} j^n), \\ &= k^{-1}\int_X \left(\sum_i h_i |s_i|^2 - \frac{w(j,k)}{k^{m+n} a_0(j)}\right)(\hat S_j -S(\xi_j)) \frac{(k\xi_j)^{m+n}}{(m+n)!}+O(k^{m+n-1} j^n).\end{align*}

We now allow $j$ to vary. Recall $$S(\xi_j) = S(\omega_b) + j^{-1}(S(\omega_B) - \Lambda_{\omega}\alpha + p(\eta)) + O(j^{-2}),$$ while $$\hat S_j = \hat S_b + j^{-1}S_{\alpha} + O(j^{-2}).$$ Thus \begin{align*}S(\xi_j) - \hat S_j &= j^{-1}(S(\omega_B) - \Lambda_{\omega}\alpha - \hat S_{\alpha} + p(\eta)) + O(j^{-2}) \\ &= j^{-1}(\Lambda_{\omega_B}\Omega -C_{\Omega} + p(\eta)) + O(j^{-2}). \end{align*}

Substituting this into the balancing energy gives \begin{align*} \mathcal{B}(\phi_{j,k}(X))  =& j^{-1}k^{-1}\int_X \left(\sum_i h_i |s_i|^2 - \frac{w(j,k)}{k^{m+n} a_0(j)}\right)(C_{\Omega} - \Lambda_{\omega_B}\Omega - p(\eta)) \frac{(k\xi_j)^{m+n}}{(m+n)!} \\ &+ O(k^{m+n-1} j^{n-1}, k^{m+n} j^{n-2}).\end{align*} Here we have used that $S(\omega_b) = \hat S_b$ to control the term involving the subleading order term in the volume form expansion of Corollary \ref{cor:volforms};
since this has one coefficient of the form $S(\omega_b) = \hat S_b$, it vanishes to leading order. This shows that in the $O(k^{2m+2n-2})$ term this expression is accurate to $O(j^{n-2})$, as claimed. Rearranging gives \begin{align*} & \int_X p(\eta)\left(\sum_i h_i |s_i|^2 - \frac{w(j,k)}{k^{m+n} a_0(j)}\right) \frac{(k\xi_j)^{m+n}}{(m+n)!} \\
=& -jk\mathcal{B}(\phi_{j,k}(X)) - \int_X (\Lambda_{\omega_B}\Omega - C_{\Omega})\left(\sum_i h_i |s_i|^2 - \frac{w(j,k)}{k^{m+n} a_0(j)}\right) \frac{(k\xi_j)^{m+n}}{(m+n)!} \\ &+ O(k^{m+n}, k^{m+n+1} j^{n-1}) . \end{align*}

Using the Bergman kernel again, we see that \begin{align*} &\int_X (\Lambda_{\omega_B}\Omega - C_{\Omega})\left(\sum_i h_i |s_i|^2 - \frac{w(j,k)}{k^{m+n} a_0(j)}\right) \frac{(k\xi_j)^{m+n}}{(m+n)!} \\ =& \int_X (\Lambda_{\omega_B}\Omega - C_{\Omega})\left(\sum_i \frac{h_i |z_i|^2}{\sum_i |z_i|^2} - \frac{w(j,k)}{h(j,k)}\right) \frac{(k\xi_j)^{m+n}}{(m+n)!} \\
&+ O(k^{m+n}, k^{m+n+1} j^{n-1}).
\end{align*}

To conclude the proof we use H\"older's inequality, which implies

\begin{align*}& \|p(\eta)\|_1 \left\|\sum_i h_i |s_i|^2_{h^k_X \otimes h^{kj}_{B}} - \frac{w(j,k)}{k^{m+n} a_0(j)}\right\|_{\infty} \\ \geq& \int_X p(\eta) \left(\sum_i h_i |s_i|^2_{h^k_X \otimes h^{kj}_{B}} - \frac{w(j,k)}{k^{m+n} a_0(j)}\right) \omega_X^m \wedge \omega_B^n, \\ =& k^{-m-n} j^{-n} m! n!  \int_X p(\eta) \left(\sum_i h_i |s_i|^2_{h^k_X \otimes h^{kj}_{B}} - \frac{w(j,k)}{k^{m+n} a_0(j)}\right) \frac{(k\xi_j)^{m+n}}{(m+n)!} \\
&+ O(1, k j^{-1}).\end{align*} Again it is important here that the coefficient of the subleading order term in the volume involving $\omega_B^{n-1}\wedge\omega_X^{m+1}$ and $  \omega_X^{m-1} \wedge \ddb \phi_R \wedge \omega_B^n$ vanishes to leading order. The result follows from summation of the equalities and inequalities we have obtained.\end{proof}

\subsection{The $L^{\infty}$-norm}\label{sec:C0norm}

Let $(\X,\H) \to (B,L)$ be a fibration degeneration. In order to bound the norm-type quantity appearing in the lower bound of Proposition \ref{prop:reduction}, it will be useful to compare the quantity occurring with Donaldson's $L^{\infty}$-norm of a test configuration \cite[p. 470]{donaldson-semistable}. The induced test configuration $(\X,\H+jL)$ with central fibre $\X_0$ and Hamiltonian $h$ has $L^{\infty}$-norm as a function of $j$ given by $$\max_i \left|\frac{\lambda_i}{k} - \frac{b_0(j)}{a_0(j)}\right |,$$ where $\lambda_i$ are the weights of the $\C^*$-action in a fixed projective embedding of the test configuration, which depend on $k$. The definition itself is independent of $k$ by Donaldson's results \cite[Section 5.1]{donaldson-semistable}.

\begin{lemma}\label{lemma:norm-bound} We have $$\left|\sum_i \frac{h_i}{k} |s_i|^2_{h^k_X \otimes h^{kj}_{B}} - \frac{w(j,k)}{k^{m+n+1}a_0(j)}\right|_{L^{\infty}} \leq \| (\X,\H+jL)\|_{\infty} + O(k^{-1},j^{-1}).$$ \end{lemma}

\begin{proof} Since we are only interested in the inequality to leading order in $k$, we may replace $w(j,k)/k^{m+n+1}$ with $b_0(j)$. We next write $$\sum_i \frac{h_i}{k} |s_i|^2_{h^k_X \otimes h^{kj}_{B}} = \sum_i \frac{h_i |z_i|^2}{k(\sum_i |z_i|^2)}\rho_{j,k}.$$ The Bergman kernel is uniformly bounded by a constant, so it is enough to obtain a bound on the $L^{\infty}$-norm of $\sum_i \frac{h_i |z_i|^2}{k(\sum_i |z_i|^2)} - b_0(j)/a_0(j)$ as a function on $X$. This is clearly bounded above by the $L^{\infty}$-norm of the same function considered on projective space, which is to say $$\left |\sum_i \frac{h_i |z_i|^2}{k(\sum_i |z_i|^2)} - \frac{b_0(j)}{a_0(j)}\right |_{L^{\infty}(X)} \leq \left | \sum_i \frac{h_i |z_i|^2}{k(\sum_i |z_i|^2)} - \frac{b_0(j)}{a_0(j)}\right|_{L^{\infty}(\pr^{N_{j,k}})}.$$

We work on the overlying projective space, and by homogeneity restrict to the set $\sum_i|z_i|^2 = 1$. Thus we are required to bound the function $$\left|\sum_i \frac{h_i}{k} |z_i|^2 - \frac{b_0(j)}{a_0(j)} \right|$$ subject to the constraint $$\sum_i|z_i|^2 = 1.$$ It is clear that the infimum is achieved with value $\inf_i \frac{h_i}{k} - \frac{b_0(j)}{a_0(j)}$ and the supremum is achieved with value $\sup_i \frac{h_i}{k} - \frac{b_0(j)}{a_0(j)}$, which precisely gives the desired result in terms of Donaldson's $L^{\infty}$-norm of a test configuration.
\end{proof}

We now prove several results concerning the norms, all of which use that the Hamiltonian associated to a fibration degeneration is independent of $j$, as established in Section \ref{sec:chern-weil}.

\begin{remark}\label{berndtsson} We expect that an analogous result to the above is true for more general norms. Related results have been proven for  geodesic segments by Berndtsson \cite{berndtsson}, however there are no known analogous results that apply to geodesic rays. Such a result would enable us to extend our lower bound on $\|p(\eta)\|_1$ to more general norms. \end{remark}

\begin{lemma}\label{lemma:c0expansion} The $L^{\infty}$-norm $\| (\X,\H+jL)\|_{\infty}$ admits an expansion $$ \| (\X,\H+jL)\|_{\infty} = c_0 + j^{-1}c_1 + O(j^{-2}).$$ \end{lemma}

\begin{proof} The maximum of the $\lambda_i/k$ is equal to the maximum of the Hamiltonian, by the definition of the Hamiltonian itself. By Proposition \ref{independent-Hamiltonian}, the Hamiltonian is independent of $j$, hence the expansion follows from the expansion of $b_0(j)/a_0(j)$. \end{proof}

The appropriate norm to consider for fibration degenerations is thus the following. 

\begin{definition}\label{infty-def} We define the $L^{\infty}$\emph{-norm} of the fibration degeneration to be $$\|(\X,\H)\|_{\infty} = c_0.$$ \end{definition}

One can analogously define other norms for fibration degenerations by the obvious generalisation, namely by using the leading order term in the expansion in $j$ of the norm of the associated test configuration. In particular, one obtains an analogous integral formulation in terms of the Hamiltonian.

\subsection{The base term} We now turn to the term $$-\int_X \left(\sum_i h_i \frac{|z_i|^2}{\sum_i |z_i|^2}- \frac{w(j,k)}{h(j,k)}\right)(\Lambda_{\omega_B}\Omega -C_{\Omega})  \omega_{FS}^{m+n}$$ involved in our lower bound on $\|p(\eta)\|_1$. By \cite[Lemma 4.1]{morefibrations}, since $\Omega$ is a form on $B$ we have \begin{equation} kj \Lambda_{\omega_{FS}}\Omega = \Lambda_{\omega_B} \Omega + O(k^{-1},j^{-1}).\end{equation} Thus if $\tilde C_{\omega}$ denotes the average value of $\Lambda_{\omega_{FS}}\Omega$ over $X$ with respect to the Fubini-Study volume form, we have $$kj\tilde C_{\omega} =  C_{\omega} + O(k^{-1},j^{-1}).$$ Thus we are interested in the integral $$-kj\int_X \left(\sum_i h_i \frac{|z_i|^2}{\sum_i |z_i|^2}- \frac{w(j,k)}{h(j,k)}\right)(\Lambda_{\omega_{FS}}\Omega - \tilde C_{\Omega})  \omega_{FS}^{m+n}.$$ We use the following result.

\begin{proposition}\cite{lejmi-szekelyhidi, twisted} Fix a smooth polarised variety $(Y,L_Y)$ of dimension $q$ and a closed semipositive integral $(1,1)$-form $\beta \in c_1(T)$ on $Y$. Consider any test configuration $(\Y,\L_{\Y})$ embedded in projective space via global sections of $L_Y^k$, with corresponding Hamiltonian $h$ with respect to the Fubini-Study metric. Then we have a lower bound $$-\int_Y \left(\sum_i h_i \frac{|z_i|^2}{\sum_i |z_i|^2}- \frac{w_Y(k)}{h_Y(k)}\right)(\Lambda_{\omega_{FS}}\Omega -  \tilde C_{\Omega})  \frac{\omega_{FS}^{q}}{q!} \geq k^q J_{T}(\Y,\L_Y) +O(k^{q-1}),$$ where $J_{T}(\Y,\L_Y)$ is the J-weight of $(\Y,\L_Y)$, defined in Equation \eqref{jweight}. \end{proposition}

The proof uses a convexity argument, and is due to Lejmi-Sz\'ekelyhidi in the case that $\beta$ is positive \cite[Lemma 8]{lejmi-szekelyhidi}. Their argument extends to the semipositive case \cite[Corollary 2.26]{twisted}; in the situation of interest to us, the form $\Omega$ is strictly semipositive on $X$ as it is pulled back from $B$. For fibration degenerations, the total space $\Y$ admits a morphism to $B$, hence one can pull back $T$ to $\Y$. In this case the J-weight is most usefully defined by the following intersection-theoretic formula \cite[Proposition 4.29]{jflow}: \begin{equation}\label{jweight}J_T(\Y,\L_Y) =  \L_Y^q.T - \frac{q}{q+1}\frac{T.L^{q-1}}{L^q}\L_Y^{q+1}.\end{equation}

\begin{lemma} Suppose $(\X,\H) \to (B,L)$ is a fibration degeneration. Then for any line bundle $T$ on $B$, we have $$J_T(\X,jL+\H) = O(j^{n-2}).$$

\end{lemma}

\begin{proof} The proof is an exercise in intersection theory. We have \begin{align*}J_T(\X,jL+\H) =& (jL+\H)^{m+n}.T - \frac{m+n}{m+n+1}\frac{(jL+H)^{m+n-1}.T}{(jL+H)^{m+n}}(jL+\H)^{m+n+1}, \\ =& j^{n-1} {m+n \choose n-1}L^{n-1}.T.\H^{m+1} \\ &- \frac{m+n}{m+n+1}\frac{(jL+H)^{m+n-1}.T}{(jL+H)^{m+n}}{m+n+1 \choose n} j^nL^n.\H^{m+1} \\ &+ O(j^{n-2}).\end{align*} To obtain the result from this, one expands and uses that $$(\H^{m+1}.(L^{n-1}.T))(L^n) = (\H^{m+1}.(L^n))(L^{n-1}.T)$$ for example since the degree of the cycle $$L^n\frac{\deg L^{n-1}.T}{\deg L^n} - L^{n-1}.T$$ on $B$ is zero, where we have included the degree operator on zero cycles for clarity. 
\end{proof}

Summing up, the base term does not affect the argument to highest order in $j$:

\begin{corollary}\label{cor:base-term} We have a lower bound $$-\int_X \left(\sum_i h_i \frac{|z_i|^2}{\sum_i |z_i|^2}- \frac{w(j,k)}{h(j,k)}\right)(\Lambda_{\omega_{B}}\Omega - C_{\Omega})  \omega_{FS}^{n+m} \geq O(k^{n+m}j^{n}, k^{n+m+1}j^{n-1}).$$\end{corollary}

It is crucial to our argument that, even though the asymptotics of this analogue of the log norm functional are $O(j^{n-1})$, we have a \emph{lower} bound rather than an \emph{upper} bound. We have no conceptual explanation for the favourable sign.

\subsection{Completing the argument} In order to complete the proof of Theorem \ref{l1-bound}, two steps remain. The first is a straightforward application of the convexity results established in Section \ref{sec:kempf-ness}. This reduces to a lower bound in terms of an integral over the central fibre $\X_0$ of the fibration degeneration, and the second step relates this integral to an algebro-geometric invariant, leading to the desired result. 

We begin with the convexity argument, which follows directly from the convexity of the balancing energy due to Luo \cite{luo} and Zhang \cite{zhang}; see also Donaldson \cite[p. 465]{donaldson-semistable}. Let $\X_0$ denote the flat limit of $\phi_{j,k}^t(X)$, where the $\C^*$-action on projective space move the embedding of $X = \phi_{j,k}^1(X)$ to $\phi_{j,k}^t(X)$ for $t \in \C^*$, so that $X = \phi_{j,k}^1(X)$. 

\begin{lemma}\label{lemma:centralfibre} We have a lower bound $$-\int_{\phi_{j,k}^1(X)} \left(\sum_ i \frac{h_i|z_i|^2}{\sum_i |z_i|^2} - \frac{w(j,k)}{h(j,k)}\right)  \omega_{FS}^{m+n}\geq - \int_{\X_0} \left(\sum_ i \frac{h_i|z_i|^2}{\sum_i |z_i|^2} - \frac{w(j,k)}{h(j,k)}\right)  \omega_{FS}^{m+n}.$$ \end{lemma}

The function $\frac{w(j,k)}{h(j,k)}$ is a constant depending only on $j,k$, so to evaluate this integral we need only consider the Hamiltonian term. Recall that $b_0(j)$ is the leading order term in $k$ of the weight polynomial $w(j,k)$.

\begin{proposition}{\cite[Proposition 3]{donaldson-semistable} \cite[Theorem 26]{wang-chow-stability}}\label{prop:b0} There is an equivariant Chern-Weil style equality  $$b_0(j)k^{n+m+1} =  \int_{\X_0}\sum_ i \frac{h_i|z_i|^2}{\sum_i |z_i|^2} \frac{\omega_{FS}^{m+n}}{(m+n)!}.$$ \end{proposition}

This is everything we need to prove Theorem \ref{l1-bound}.

\begin{proof}[Proof of Theorem \ref{l1-bound}]

Proposition \ref{prop:reduction} produces a lower bound on $$\|p(\eta)\|_1 \left \|\sum_i h_i |s_i|^2_{h^k_X \otimes h^{kj}_{B}} - \frac{w(j,k)}{k^{m+n} a_0(j)}\right\|_{\infty}$$  in terms of 
\begin{align*} &-\binom{m+n}{n}^{-1} k^{1-m-n} j^{1-n}  \int_X \left(\sum_ i \frac{h_i|z_i|^2}{\sum_i |z_i|^2} - \frac{w(j,k)}{h(j,k)}\right) \omega_{FS}^{m+n}  \\ &- \binom{m+n}{n}^{-1}  k^{-m-n} j^{-n} \int_X \left(\sum_i h_i \frac{|z_i|^2}{\sum_i |z_i|^2}- \frac{w(j,k)}{h(j,k)}\right)(\Lambda_{\omega_B}\Omega -C_{\Omega})  \omega_{FS}^{m+n}. \end{align*}
Lemma \ref{lemma:centralfibre} shows  that the first term of these two terms is bounded below by $$-\binom{m+n}{n}^{-1} k^{1-m-n} j^{-n} \int_{\X_0} \left(\sum_ i \frac{h_i|z_i|^2}{\sum_i |z_i|^2} - \frac{w(j,k)}{h(j,k)}\right)  \omega_{FS}^{m+n},$$ while from Corollary \ref{cor:base-term} the second is bounded below by $O( 1,k j^{-1})$. From Proposition \ref{prop:b0}, the integral $ -\int_{\X_0} \left(\sum_ i \frac{h_i|z_i|^2}{\sum_i |z_i|^2} - \frac{w(j,k)}{h(j,k)}\right)  \frac{\omega_{FS}^{m+n}}{(m+n)!}$ equals $$-k^{m+n+1} b_{0} (j) + k^{m+n}\frac{w(j,k)}{h(j,k)}a_{0}(j) = -k^{m+n+1} \left(b_{0} (j) - \frac{w(j,k)}{k h(j,k)}a_{0}(j) \right) ,$$ where we have used that $$a_{0} (j)k^{m+n}= \int_X \frac{\omega_{FS}^{m+n}}{(m+n)!}.$$ Combining these bounds and expanding the various polynomials, it follows that \begin{align*}&\|p(\eta)\|_1 \left\|\sum_i h_i |s_i|^2_{h^k_X \otimes h^{kj}_{B}} - \frac{w(j,k)}{k^{m+n} a_0(j)}\right\|_{\infty} \\ &\geq - m!n! k^{2} j^{1-n} \left(b_{0} (j) - \frac{w(j,k)}{kh(j,k)}a_{0}(j) \right) + O(1, kj^{-1}) \\ &=-m!n! k j^{1-n} \left( \frac{b_0 (j) a_1(j)}{a_0(j)} - b_1(j) \right) + O(1, k j^{-1}) \\ &=-m!n!  k \left( j\W_0(\X,\H)+\W_1(\X,\H) \right)+ O(1, kj^{-1}).\end{align*}

Lemma \ref{lemma:norm-bound} gives an upper bound $$\left|\sum_i h_i |s_i|^2_{h^k_X \otimes h^{kj}_{B}} - \frac{w(j,k)}{k^{m+n}a_0(j)}\right|_{L^{\infty}} \leq k \| (\X,\H+jL)\|_{\infty},$$ hence implying $$\|p(\eta)\|_1 \geq - m!n!  \frac{j\W_0(\X,\H)+\W_1(\X,\H)}{\|(\X,\H)\|_{\infty}}.$$ The result follows as the choice of $\omega_X$ and fibration degeneration $(\X,\H) \to (B,L)$ were arbitrary, meaning one can take an infimum on the left hand side and a supremum the right hand side.
\end{proof}

\begin{remark} A key element of our proof of Theorem \ref{l1-bound} was the convexity of the usual balancing energy. In particular, we did not rely on the finite-dimensional results of Section \ref{sec:kempf-ness}. Thus it is natural to ask how the finite dimensional geometry of fibrations is related to the ``infinite dimensional'' K\"ahler geometry. This happens in three ways. Firstly, and most importantly, the key notion of a fibration degeneration is precisely motivated by the results of Section \ref{sec:kempf-ness}, where they are shown to arise naturally in GIT. Secondly, the numerical invariant defined to be the GIT weight in Definition \ref{stability-in-finite-dims}  plays exactly the same role in our work as the Chow weight plays in Donaldson's work on K-stability \cite{donaldson-semistable}. Finally, while it seems challenging to directly use the convexity of the moment map produced in Section \ref{sec:kempf-ness}, it is also interesting to note that the associated energy functional appears formally as the leading order term in an expansion of the balancing energy: \begin{align*}\int_{\phi_{j,k}^t(X)} &\left(\sum_ i \frac{h_i|z_i|^2}{\sum_i |z_i|^2} - \frac{w(j,k)}{h(j,k)}\right)  \omega_{FS}^{m+n} \\ & = k^m\int_{\phi_{j,k}^t(X)}  \left(\sum_ i \frac{h_i|z_i|^2}{\sum_i |z_i|^2} - \frac{w(j,k)}{h(j,k)}\right)  \omega_{FS}^m \wedge \omega_B^n + O(k^{m-1}j^n, k^mj^{n-1}).\end{align*} Thus, to leading order, the balancing energy precisely equals the energy functional associated to a moment map in finite dimensions, giving a further conceptual explanation for Theorem \ref{l1-bound} and its links to the results of Section \ref{sec:kempf-ness}.

\end{remark}

\bibliography{stablefibrations}
\bibliographystyle{amsplain}

\end{document}